\documentclass[a4paper,11pt]{article}
\usepackage[utf8]{inputenc}
\usepackage{latexsym}
\usepackage{amsfonts}
\usepackage{amssymb}

\usepackage{default}
\usepackage[T1]{fontenc}
\usepackage{amsmath}
\usepackage{amssymb}
\usepackage{amsfonts}
\usepackage{amsthm}
\usepackage{indentfirst}
\usepackage{psfrag}
\usepackage{amscd,stmaryrd,latexsym}
\usepackage[margin=1.3in]{geometry}
\usepackage{graphicx}
\newtheorem{thm}{Theorem}[section]
\newtheorem{lem}{Lemma}[section]

\newtheorem{Prop}{Proposition}
\newtheorem*{St*}{Statement}

\newtheorem*{theorem*}{Theorem}
\theoremstyle{remark}

\theoremstyle{definition}

\theoremstyle{remark}
\newtheorem{oss}{Remark}[section]

\newcommand{\be}{\begin{equation}}
\newcommand{\ee}{\end{equation}}
\newcommand{\R}{\mathbb{R}}
\newcommand{\N}{\mathbb{N}}

\newcommand{\C}{\mathbb{C}}
\newcommand{\CP}{\mathbb{C}\mathbb{P}}
\newcommand{\Z}{\mathbb{Z}}

\newcommand\res{\mathop{\hbox{\vrule height 7pt width .5pt depth 0pt
\vrule height .5pt width 6pt depth 0pt}}\nolimits}

\def\eps{\mathop{\varepsilon}}
\def\Mc{\mathop{\mathcal{M}_{}}}

\def\Uc{\mathop{\mathcal{U}_{}}}

\def\Nc{\mathop{\mathcal{N}_{}}}

\def\Hc{\mathop{\mathcal{H}}}

\def\de{\delta}

\def\Om{\Omega}
\def\om{\omega}

\def\p{\partial}

\def\atanl{\left(\alpha_{\ell}^{2m-1}\right)_{\text{tan}}}
\def\atan[#1]{\left(\alpha_{#1}^{2m-1}\right)_{\text{tan}}}
\def\j[#1]{\mathscr{j_{#1}}}

\def\eps{\mathop{\varepsilon}}

\def\Om{\Omega}
\def\om{\omega}
\def\p{\partial}

\DeclareMathAlphabet{\mathscr}{OT1}{pzc}{m}{it} 

\begin{document} 

\title{\textbf{Compactness results for triholomorphic maps}}
\author{Costante Bellettini \\ University of Cambridge \and Gang Tian \\ Beijing University and Princeton University}
\date{}
\maketitle
 \begin{abstract}
  
 We consider triholomorphic maps from an almost hyper-Hermitian manifold $\Mc^{4m}$ into a (simply connected) hyperK\"ahler manifold $\Nc^{4n}$. This notion entails that the map $u \in W^{1,2}$ satisfies a quaternionic del-bar equation. We work under the assumption that $u$ is locally strongly approximable in $W^{1,2}$ by smooth maps: then such maps are almost stationary harmonic, in a suitable sense (in the important special case that $\Mc$ is hyperK\"ahler as well, then they are stationary harmonic). We show, by means of the $bmo-\mathscr{h}^1$-duality, that in this more general situation the classical $\eps$-regularity result still holds and we establish the validity, for triholomorphic maps, of the $W^{2,1}$-conjecture (i.e. an a priori $W^{2,1}$-estimate in terms of the energy). We then address compactness issues for a weakly converging sequence $u_\ell \rightharpoonup u_\infty$ of strongly approximable triholomorphic maps $u_\ell:\Mc \to \Nc$ with uniformly bounded Dirichlet energies. The blow up analysis leads, as in the usual stationary setting, to the existence of a rectifiable blow-up set $\Sigma$ of codimension $2$, away from which the sequence converges strongly. The defect measure $\Theta(x) {\Hc}^{4m-2}\res \Sigma$ encodes the loss of energy in the limit and we prove that for a.e. point on $\Sigma$ the value of $\Theta$ is given by the sum of the energies of a (finite) number of smooth non-constant holomorphic bubbles (here the holomorphicity is to be understood with respect to a complex structure on $\Nc$ that depends on the chosen point on $\Sigma$). In the case that $\Mc$ is hyperK\"ahler this quantization result was established by C. Y. Wang \cite{Wang} with a different proof; our arguments rely on Lorentz spaces estimates. By means of a calibration argument and a homological argument we further prove that whenever the restriction of $\Sigma \cap (\Mc \setminus \text{Sing}_{u_\infty})$ to an open set is covered by a Lipschitz connected graph, then actually this portion of $\Sigma$ is a smooth submanifold without boundary and it is pseudo-holomorphic for a (unique) almost complex structure on $\Mc$ (with $\Theta$ constant on this portion); moreover the bubbles originating at points of such a smooth piece are all holomorphic for a common complex structure on $\Nc$. 
 \end{abstract}
 
\section{Introduction}

Triholomorphic maps have appeared in high-dimensional gauge theory and string theory \cite{FKS} and play an important role for the geometric understanding of hyperK\"ahler manifolds \cite{Tian2}. From a more analytic perspective, they constitute an important class of harmonic maps and their more peculiar structure may lead to deeper results than those coming from the general high-dimensional harmonic map theory.

Consider two compact oriented Riemannian manifolds $(\Mc, g)$ and $(\Nc, h)$ respectively of dimensions $4m$ and $4n$ and assume that the manifold $\Mc$ [respectively $\Nc$] carries three smooth almost complex structures $i$, $j$, $k$ compatible with the metric $g$ [respectively $I$,$J$,$K$ compatible with $h$] satisfying the quaternionic relation $ijk =-Id$ [respectively $IJK =-Id$]. To each of these almost complex structures, take for example $i$, we can associate uniquely a non-degenerate two-form $\om_i$ on $\Mc$ defined by $\om_i(\cdot, \cdot)=g(\cdot, i \cdot)$. In the same way we have $\om_j$ and $\om_k$ on $\Mc$ and (using the metric $h$ instead) $\Om_I$, $\Om_J$, $\Om_K$ on $\Nc$. In the case when $\Mc$ and $\Nc$ are hyperK\"ahler manifolds all the almost complex structures are actually parallel complex structures and all of the two-forms defined above are closed. We will require, for the purposes of this work, that the forms $\Om_I$, $\Om_J$, $\Om_K$ on $\Nc$ are closed, whilst we will not need the closedness for $\om_i$, $\om_j$ and $\om_k$ on $\Mc$. A manifold $\Mc$ of the type just described is usually called \textit{almost hyper-Hermitian}. By a result of N. Hitchin \cite[Lemma 6.8]{Hitchin} the closedness assumption on the forms on $\Nc$ amounts to the apparently stronger condition that $\Nc$ is a hyperK\"ahler manifold. We will always assume that $\Nc$ is simply connected\footnote{The assumption that the hyperK\"ahler manifold $\Nc$ is simply connected is not very restrictive: indeed, when it is not so, the manifold admits a finite cover that is the product of several simply connected hyperK\"ahler manifolds with a flat torus $\mathbb{T}^{4n}$. The flat torus does not admit any harmonic $S^2$'s so for the purposes of the present work, where we will deal with compactness issues, it represents a very easy case as it does not allow bubbles to form.}.

By the Nash embedding theorem we can isometrically embed $\Nc$ in $\R^Q$ for some large enough $Q$, and from now on we will assume to have done this. We say that $u:\Mc \to \Nc$ is $W^{1,2}(\Mc, \Nc)$ if $u$ is of class $W^{1,2}$ as a map from $\Mc$ into $\R^Q$ and moreover $u(x) \in \Nc$ for $\mathcal{H}^{4m}$-a.e. $x \in \Mc$. We consider $u$ (and later on a sequence $\{u_\ell\}$) of class $W^{1,2}(\Mc, \Nc)$ satisfying the triholomorphic\footnote{With our convention for the PDE (\ref{eq:triholo}) it is clear that $(i,-I)$-pseudoholomorphic maps are special cases of triholomorphic maps. The same goes more generally for $((ai+bj+ck), -(aI+bJ+cK))$-pseudo holomorphic maps, with $a^2+b^2+c^2=1$. However, as noted with explicit examples in \cite{ChenLi1}, the triholomorphic notion is stricly more general than the pseudoholomorphic one.} map equation

\begin{equation}
 \label{eq:triholo}
 du = I du \, i + J du \, j + K du \, k,
\end{equation}
i.e. for a.e. $x\in \Mc$ and any tangent vector $X$ at $x$ it holds

$$ du (X) = I du (i X) + J du (j X) + K du (k X),$$
where $I,J,K$ act on $T_{u(x)}\Nc$. 
The situation that we will address is the following. We assume that the map $u$ satisfies the following\footnote{Actually we will need  (\ref{eq:strongapproximability}) only for $\alpha = \Om_I, \Om_J, \Om_K$, and these three forms do not generate the whole of $H^2(\Nc,\Z)$, however we are not sure whether this actually yields a weaker assumption or whether the fulfilment of (\ref{eq:strongapproximability}) by a triholomorphic map is implied by its validity on these three forms only.} condition

\begin{equation}
 \label{eq:strongapproximability}
d(u^*\alpha)=0 \text{ for every closed $2$-form } \alpha  \text{ on }\Nc.
\end{equation}
Condition (\ref{eq:strongapproximability}) is rather easily seen to be true in the important case that $u$ is \textit{locally strongly approximable} in $W^{1,2}$ by smooth maps. In view of \cite{BCDH} we have that (\ref{eq:strongapproximability}) is actually equivalent to the local strong approximability when $H_2(\Nc,\Z)$ is torsion-free. This is satisfied for example when $n=1$, i.e. $\Nc$ is a K3 surface. We are not aware of any example with non-trivial torsion in $H_2(\Nc, \Z)$.

\medskip

A map $u$ as above turns out to be stationary harmonic (i.e. a critical point of the Dirichlet energy for local perturbations both in the target and in the domain, this notion was introduced by R. Schoen \cite{Schoen}) when the forms $\om_i$, $\om_j$, $\om_k$ are closed, i.e. when $\Mc$ as well is hyperK\"ahler. The assumption (\ref{eq:strongapproximability}) is crucial here\footnote{This was overlooked in \cite{ChenLi1}.} as we will see. In the more general case that we address (i.e. $\om_i$, $\om_j$, $\om_k$ are not necessarily closed) then we find a suitable notion of ``almost stationarity'': for any vector field $X \in C^\infty_c(\Mc)$ and denoting with $\Psi_t=Id + t X$ the local variation induced by $X$ then we have
$$\left.\frac{d}{dt}\right|_{t=0}\int\limits_{\Mc}|\nabla (u \circ \Psi_t)|^2 \geq -C \int\limits_{\Mc} |X| |\nabla u|^2,$$
where $C=\|d\om_i\|_\infty+\|d\om_j\|_\infty+\|d\om_k\|_\infty$. A similar statement holds for local perturbations in the target. We observe further that $u$ satisfies a \textit{perturbation of the usual harmonic map PDE}, namely (denoting by $A$ the second fundamental form of $\Nc$)
\begin{equation}
\label{eq:2ndorderPDE}
\Delta_g u + A(\nabla u , \nabla u) = f(x,u,\nabla u)    \, ,                                                                                                       \end{equation}
where $f$ is smooth and has linear dependence on the first derivatives of $u$ (observe that this type of perturbation is \textit{sub-critical}). Moreover we prove in Proposition \ref{Prop:monotonicity} an \textit{almost monotonicity formula} for the energy ratio $\frac{1}{r^{4m-2}}\int_{B_r(x)}|\nabla u|^2$ at all $x\in\Mc$.

\medskip

Inspired by the similarities with stationary harmonic maps, we show that the basic results that are known in that classic setting are still valid in our situation, namely we prove an \textit{$\eps$-regularity result} (see Propositions \ref{Prop:epsregularity} and \ref{Prop:singset}), i.e. there exists a threshold $\eps_0>0$ depending only on the geometric data (and not on the particular $u$ satisfying (\ref{eq:triholo}) and (\ref{eq:strongapproximability})) such that whenever the energy ratio of $u$ is below $\eps_0$ on a ball $B_R(x)$ then $u$ is smooth on $B_{R/2}(x)$ with a $\sup$-bound on the modulus of its gradient. Moreover the singular set of $u$, i.e. the closed set where $u$ fails to be smooth, is a ${\Hc}^{4m-2}$-negligeable set. These results are based on the PDE analysis carried out in Sections \ref{jacobianstructure} and \ref{LaplaceBeltrami}, where we find that the PDE (\ref{eq:2ndorderPDE}) is of the form $\Delta_g u \in \mathscr{h}^1$ (where $\mathscr{h}^1$ denotes the inhomogeneous Hardy space) with Hardy norm controlled by the Dirichlet energy of $u$. Then we use the $bmo-\mathscr{h}^1$ duality, similarly to the case treated by L. C. Evans \cite{Evans}, to infer the $\eps$-regularity\footnote{Evans proved the result for stationary harmonic maps when the target is a sphere. For stationary harmonic maps with arbitrary target the $\eps$-regularity result was established by F. Bethuel \cite{Beth}. For minimizing harmonic maps it had been earlier proved by R. Schoen and K. Uhlenbeck \cite{SchUhl}.}. We also establish the \textit{$W^{2,1}$-conjecture for triholomorphic maps}, i.e. the analogue of the a priori $W^{2,1}$-estimate conjectured to hold by T. Rivi\`{e}re \cite[page 5]{Riv} in the case of stationary harmonic maps:

\begin{thm}
 \label{thm:W21estimate}
Let $u: \Mc \rightarrow \Nc \subset \R^Q$ be a triholomorphic map, i.e. (\ref{eq:triholo}) holds. Then
$$u \in W^{2,1}(\Mc) \text{ and } \|u\|_{W^{2,1}(\Mc)} \leq Const_{(\Mc, g)} \left(Const_{(\Nc \hookrightarrow \R^Q)} + E(u)\right),$$
where $E(u)$ denotes the Dirichlet energy of $u$ and $Const_{(\Nc \hookrightarrow \R^Q)}$ is a constant depending on the embedding of $\Nc$ and $Const_{(\Mc, g)}$ depends only on the geometry of $\Mc$.
\end{thm}

This will be a fundamental ingredient in the proof of our quantization result Theorem \ref{thm:quantization}. In the case of stationary harmonic maps from a domain in $\R^n$ into a spherical target the $W^{2,1}$-estimate is known to hold thanks to the observation by F. H\'{e}lein that $\Delta u \in \mathscr{H}^1$, the homogeneous Hardy space (see \cite{CLMS}, \cite{Helein}). This a priori estimate is the starting point of the proof in \cite{Evans}, where the $BMO-\mathscr{H}^1$-duality is exploited.

\medskip

After these preliminary results we proceed further with compactness questions for a sequence $\{u_\ell\}$ of triholomorphic maps (with assumptions (\ref{eq:triholo}) and (\ref{eq:strongapproximability}) as above for each $u_\ell$) with uniformly bounded Dirichlet energies. Without loss of generality we may assume $u_\ell \rightharpoonup u$ weakly in $W^{1,2}$. It is natural to ask what knowledge we have of the limit map $u$ and of its smoothness properties and if (and on which set) we may expect a stronger form of convergence. The question is wide and open to a large extent (even in the case where both $\Mc$ and $\Nc$ are hyperK\"ahler) and its solution would immediately impact the \textit{compactification of the moduli space of triholomorphic maps}, leading further to the definition of geometric invariants of hyperK\"ahler manifolds by means of a suitable counting of triholomorphic maps, as outlined by the second author in \cite{Tian2}. The identification of a suitable class of triholomorphic maps is one of the tasks here, and it is also in this light that condition (\ref{eq:strongapproximability}) should be viewed. In Section \ref{blowupanalysis} we show that the fundamental results of F.-H. Lin \cite{Lin} for stationary harmonic maps hold with corresponding statements in our case, namely: we can identify a closed set $\Sigma$ of locally finite ${\Hc}^{4m-2}$-measure away from which the convergence $\{u_\ell\} \to u$ is strong (even $C^k$), the set $\Sigma$ is $(4m-2)$-rectifiable, the Radon measures $|\nabla u_\ell|^2 d{\Hc}^{4m}$ weakly-* converge to the measure $|\nabla u|^2 d{\Hc}^{4m} + \Theta(x) d{\Hc}^{4m-2} \res \Sigma$ with density $\Theta \geq \eps_0$ a.e.\footnote{The measure $\Theta(x) d{\Hc}^{4m-2} \res \Sigma$ is usually called \textit{defect measure} and encodes the possible loss of energy in the limit and lack of strong convergence. The set $\Sigma$ is usually called \textit{blow-up set}, \textit{bubbling set} or \textit{concentration set}. It is the set where the energy concentrates, leading to lack of strong convergence and formation of bubbles.} Moreover $u$ is smooth away from a closed set $\text{Sing}_u \subset \Sigma$ (on which it is discontinuous) and satisfies globally the PDE $\Delta_g u + A(\nabla u , \nabla u) = f(x,u,\nabla u)$. However, whilst $u$ is smooth away from $\text{Sing}_u$ (and thus $u$ is almost stationary on $\Mc \setminus \text{Sing}_u$) the same might a priori not hold on the whole of $\Mc$. A priori we could have that (for some closed $2$-form $\alpha$) the form $d(u^* \alpha)$ is a non-zero distribution with support on the set $\text{Sing}_u$, since the weak $W^{1,2}$-convergence does not allow to pass condition (\ref{eq:strongapproximability}) to the limit. The understanding of whether the local strong approximability holds across $\text{Sing}_u$ is of central importance\footnote{Unfortunately \cite{ChenLi1} does not notice that it plays a role in the analysis of the regularity of $\Sigma$.}. It is very much related to the understanding of the structure of $\Sigma$ as well. To see this, let us focus on the important case where $\Mc$ and $\Nc$ are hyperK\"ahler: the lack of the strong approximability condition would imply a lack of stationarity of $u$ and (by the blow up formula in \cite{LiTian}) a lack of stationarity for the varifold $\Sigma$, localized on $\text{Sing}_u$, with obvious implications on the regularity. Remark that a priori $\Sigma$ might as well have boundary contained in $\text{Sing}_u$. 

\medskip

The main compactness results we provide are the following two theorems. The first regards the \textit{quantization of energy} for the defect measure:

\begin{thm}
\label{thm:quantization}
Consider a sequence of maps $u_\ell :\Mc \to \Nc$ (with $\Mc$ almost hyper-Hermitian and $\Nc$ hyperK\"ahler of arbitrary dimensions respectively $4m$ and $4n$) satisfying (\ref{eq:triholo}) and (\ref{eq:strongapproximability}), with uniformly bounded Dirichlet energies. Let $u_\ell \rightharpoonup u$ weakly in $W^{1,2}$ and let $\Sigma$ be the blow-up set, i.e. $|\nabla u|^2 d{\Hc}^{4m} \rightharpoonup |\nabla u|^2 d{\Hc}^{4m} + \Theta(x) d{\Hc}^{4m-2} \res \Sigma$ with $\Theta \geq \eps_0>0$. Then for ${\Hc}^{4m-2}$-a.e. $x \in \Sigma$ we have 
\begin{equation}
\label{eq:sumofbubbles}                                             
\Theta(x)=\sum_{s=1}^{S_x} E(\phi_s),                                                                                                                                                                                                                                                                                                                                                                                                           \end{equation}
where $S_x \in \N$ and each $\phi_s:S^2 \to \Nc$ is a smooth non-constant harmonic map (these are called \textit{bubbles}).
\end{thm}
In other words we prove that the whole loss of energy encoded in the defect measure comes from bubbling off of $S^2$'s and there is no leftover energy in the necks connecting the bubbles\footnote{A quantization result of this type is not known for arbitrary stationary harmonic maps, in \cite{LR} the target is a sphere.}. For ${\Hc}^{4m-2}$-a.e. $x \in \Sigma$ we further have that the bubbles are \textit{holomorphic} $S^2$'s for a certain choice of almost complex structures $j_x$ and $J_x$ (depending on the point $x$!), see Proposition \ref{Prop:complexstructureofbubble}. This quantization result was proved, in the case where both $\Mc$ and $\Nc$ are hyperK\"ahler, by C.-Y. Wang \cite{Wang}. In the present work we extend that result to our more general geometric setting, using a different proof inspired by the techniques of Lin-Rivi\`{e}re \cite{LR}. In Remark \ref{oss:Walpuskipaper} and Section \ref{Fueter} we describe a \textit{gauge theoretical application} where the more general setting is required.

%Eells Sampson 64 Sec curvature $\leq 0 \Rightarrow $ no bubbles
\medskip

Finer properties blow-up sets are generally a very difficult task and not much is known beyond rectifiability. With the second compactness result we start, in the case of interest, a more detailed analysis of the \textit{structure of the blow-up set} $\Sigma$ and provide a first important step towards the understanding of it. Roughly speaking we show, by means of a new idea based on a homological argument and a calibration argument (and using the quantization result), that whenever $\Sigma \setminus \text{Sing}_u$ has a bit of regularity, then it is pseudoholomorphic for a fixed almost complex structure and the bubbles produced are holomorphic spheres for a fixed complex structure in $\Nc$. Precisely we have:

\begin{thm}
\label{thm:localholomorphicitysmooth}
Under the same hypotheses of Theorem \ref{thm:quantization}, whenever we have an open ball $B^{4m} \subset \Mc \setminus \text{Sing}_u$ such that $\Sigma \cap B^{4m}$ is contained in a (connected) boundaryless Lipschitz $(4m-2)$-submanifold\footnote{Here we mean that for every point on the submanifold there exists an open cylinder in which the submanifold can be expressed as the graph of a Lipschitz function from $B_\rho^{4m-2} \to \R^2$ (upon rotating and relabelling the coordinate axes).} $\mathcal{L} \subset B^{4m}$ and $\Hc^{4m-2}(\Sigma \cap B^{4m})>0$ then there exist constants $a,b,c$ with $a^2 + b^2 + c^2 =1$ such that 
\begin{description}
 \item[(i)] $\Sigma \cap B^{4m}$ coincides with the whole of $\mathcal{L}$ and it is pseudo holomorphic with respect to the almost complex structure $ai +bj +ck$ (a posteriori this pseudo holomorphicity implies that $\mathcal{L}$ was actually smoooth);
 \item[(ii)] for every $x \in \Sigma  \cap B^{4m}$ the bubbles $(\phi_s)_*(S^2)$ that appear in (\ref{eq:sumofbubbles}) are holomorphic with respect to the complex structure  $-(aI +bJ +cK)$ on $\Nc$. Moreover $\Theta(x)$ in (\ref{eq:sumofbubbles}) is constant on $\Sigma \cap B^{4m}$.
\end{description}
\end{thm}

\begin{oss}
The result of \cite{LiTian} gives that, in the case when $\Mc$ is also hyperK\"ahler (and thus the maps $u_\ell$ are stationary harmonic), the varifold $\Theta(x) d{\Hc}^{4m-2} \res \Sigma$ is stationary away from $\text{Sing}_u$: well-known GMT results (see Allard's theorem \cite{Allard}) yield that for a stationary varifold there exists a dense open subset where the varifold is smooth (but stronger regularity properties for arbitrary stationary varifolds are unknown, although expected). In particular, whenever $u$ has a continuity point on $\Sigma$, the assumption on $\Sigma$ in Theorem \ref{thm:localholomorphicitysmooth} is somewhere satisfied. The latter theorem thus yields that each connected smooth piece of $\Sigma$ is actually \textit{(pseudo)-holomorphic} (of codimension $2$ in $\Mc$). The point here is that, for each such connected component, we have a fixed almost complex structure: this also implies that for all points in such a connected component the almost complex structure for which the bubbles are holomorphic is \textit{independent of the point} (compare the weaker statement mentioned after the quantization result). To our knowledge this is the first instance of application of a quantization result for stationary harmonic maps in a high-dimensional situation. 
\end{oss}

We cannot prove that the whole of $\Sigma$ needs to be (pseudo)-holomorphic for a unique almost complex structure. A priori it might be conceivable, for example, that several smooth pieces (each one pseudo holomorphic for a different structure) come together along a submanifold of codimension $3$ that is contained in $\text{Sing}_u$. Ruling out such a behaviour would of course impact the knowledge of the regularity properties of $u$. This is an aspect of the deep connections between the stationarity properties of $u$ and the regularity properties of $u$ and $\Sigma$. We conjecture the following.

\textbf{Conjecture}: With the same assumptions as in Theorems \ref{thm:quantization} and \ref{thm:localholomorphicitysmooth}, $\text{Sing}_u$ has codimension at least $3$ and $\Sigma$ is made of a finite number of (pseudo)-holomorphic varieties (each connected component is holomorphic for a certain almost complex structure on $\Mc$, possibly all of these structures might be the same), each of which is a cycle in $\Mc \setminus \text{Sing}_u$ and whose common boundary is the $(4m-3)$-dimensional stratum of $\text{Sing}_u$. 

\begin{oss}
A very important problem to address is the existence of a homogeneous triholomorphic map with a singular set of codimension $3$. The only known example is constructed in the case of a non-compact target (this is shown in \cite{ChenLi2} based on the construction of \cite{Atiyah Hitchin}). The non-existence of such maps would impact the previous conjecture, which could be strenghtned considerably: all varieties would be cycles and $\text{Sing}_u$ would have codimension at least $4$. 
\end{oss}

\begin{oss}
Another important question to address, on which however we do not focus at all in this work, is the optimal estimate on the size of the singular set for a triholomorphic map satisfying (\ref{eq:strongapproximability}): is it possible to show that the codimension is at least $4$? We wish to stress here that the example of triholomorphic map with singular set of codimension $3$ mentioned in the previous remark is stationary harmonic but does not satisfy (\ref{eq:strongapproximability}).
\end{oss}

\begin{oss}
We believe that the argument given in Section \ref{Holomorphicity properties of smooth blow-up sets} for the proof of Theorem \ref{thm:localholomorphicitysmooth} is robust enough to be pushed to a situation with weaker assumptions, however for the moment we prefer to limit ourself to the given statement and postpone any improvement to future work.
\end{oss}

\begin{oss}
\label{oss:Walpuskipaper}
As we described above, we extend the triholomorphic notion to the case where $\Mc$ is almost hyper-Hermitian (rather than hyperK\"ahler). We then show that the first order PDE (\ref{eq:triholo}) and the geometric structures on the manifolds $\Mc$ and $\Nc$ lead to a jacobian structure for $\Delta_g u$ which enables the use of Hardy space techniques (this plays an important role both in the $\eps$-regularity result and in the quantization result). Relaxing the assumptions on $\Mc$ is interesting not only as an example of ``almost stationary harmonic maps'', as we described earlier. Indeed the analytic and geometric techniques employed in the proofs of our results can find a rather direct application in the problem treated e.g. by T. Walpuski \cite{Walpuski}. In that work the author deals with triholomorphic sections in a bundle of hyperK\"ahler manifolds (with $m=1$, using our notations); the problem originates in Gauge Theory on $Spin(7)$ manifolds. The PDE satisfied by these triholomorphic sections corresponds to our equation (\ref{eq:triholo}) with in addition lower order perturbation terms (depending on $u$ but not on its derivatives). Our analysis can be carried out in the same fashion in the application under consideration, in particular it gives an affirmative answer to the quantization question, which is conjectured to be true but left open in \cite[page 5]{Walpuski}. We postpone to the final section a brief explanation of the modifications involved.
\end{oss}

\medskip

\textit{\textbf{Acknowledgements}}: The first author wishes to thank Tristan Rivi\`{e}re for very fruitful conversations while this work was in progress. The authors wish to thank Thomas Walpuski for bringing their attention to the topic of Fueter sections, described in Section \ref{Fueter}.

\section{Dirichlet energy and almost-monotonicity formula}
\label{Energy&Monotonicity}
We begin this section by recalling (see \cite{ChenLi1}, \cite{LiTian}) that the Dirichlet energy density for a triholomorphic map can be written, thanks to (\ref{eq:triholo}), as 
\begin{equation}
 \label{eq:DirEnergyTriholo}
\frac{1}{2}|\nabla u|_g^2=-\frac{1}{(2m-1)!}\left[ (\om_i)^{2m-1} \wedge u^* \Om_I + (\om_j)^{2m-1} \wedge u^* \Om_J + (\om_k)^{2m-1} \wedge u^* \Om_K\right].
\end{equation}

This holds pointwise a.e. on the domain for $u \in W^{1,2}$ and it suffices to prove it at a point $x$ by choosing coordinates so that $(T_x \Mc, i, j, k)$ is isometrically identified with $(\mathbb{H}^m, \mathscr{j_1}, \mathscr{j_2}, \mathscr{j_3})$, where $\{\mathscr{j_\ell}\}_{\ell=1}^3$ are the usual quaternion units and $\{\alpha_\ell\}_{\ell=1}^3$ are the standard associated K\"ahler forms. Consider a standard orthinormal basis $\{e_1, \j[1] e_1, \j[2] e_1, \j[3] e_1, ... , e_m, \j[1] e_m, \j[2] e_m, \j[3] e_m\}$. The differential $du$ is then a linear map into $T_{u(x)}\Nc$ satisfying $du =  I du \, \mathscr{j_1} + J du \, \mathscr{j_2} + K du \,\mathscr{j_3}$. The computation in \cite{ChenLi1} Proposition 2.2 shows that
for any map $v$, not necessarily triholomorphic, it holds

\begin{equation}
\label{eq:triholoinequality}
-\frac{1}{(2m-1)!}\left[ (\alpha_1)^{2m-1} \wedge v^* \Om_I + (\alpha_2)^{2m-1} \wedge v^* \Om_J + (\alpha_3)^{2m-1} \wedge v^* \Om_K\right]=$$ $$=\frac{1}{2}|d v|^2 - \frac{1}{4}|dv -  I dv \, \mathscr{j_1} - J dv \, \mathscr{j_2} - K dv \,\mathscr{j_3}|^2 .
\end{equation}
and in particular, for $u$ triholomorphic, we find

$$-\frac{1}{(2m-1)!}\left[ (\alpha_1)^{2m-1} \wedge u^* \Om_I + (\alpha_2)^{2m-1} \wedge u^* \Om_J + (\alpha_3)^{2m-1} \wedge u^* \Om_K\right]=$$ $$=\frac{1}{2}|d u|^2. $$

\medskip

The significance of (\ref{eq:triholoinequality}) and (\ref{eq:DirEnergyTriholo}) relies in the following fact. Under the assumptions 
\begin{description}
 \item[(a)]  $\Mc$ and $\Nc$ are Hyperk\"ahler (so $\om_i$, $\om_j$, $\om_k$, $\Om_I$, $\Om_J$, $\Om_K$ are closed)
 \item[(b)] whenever $\alpha$ is a closed $2$-form on $\Nc$ then $d(u^* \alpha)=0$, cf. (\ref{eq:strongapproximability})
\end{description}
the quantity $-\frac{1}{(2m-1)!}\int\limits_{\Mc}\left[ (\om_i)^{2m-1} \wedge u^* \Om_I + (\om_j)^{2m-1} \wedge u^* \Om_J + (\om_k)^{2m-1} \wedge u^* \Om_K\right]$ is a \textit{null Lagrangian} for local variations in the target and in the domain (see e.g. \cite{RT1} for the analogous statement and proof in the case of pseudoholomorphic maps). Then under such variations, for $u$ triholomorphic, in view of (\ref{eq:triholoinequality}) we can see that $u$ is \textit{weakly harmonic and stationary harmonic}. In the case that $u$ is a \textit{smooth} triholomorphic map it follows from (\ref{eq:DirEnergyTriholo}) that $u$ minimizes the Dirichlet energy in its homotopy class.

We will be interested in a more general setting, namely we will drop the assumptions that  $\om_i$, $\om_j$, $\om_k$ are closed. Then the same argument shows that a triholomorphic map $u$ is \textit{almost stationary harmonic}, in the sense specified in (\ref{eq:almoststationary}) below.

Let $X \in C^\infty_c$ be a vector field on $\Mc$ and $\Psi_t = Id + tX$ be the corresponding $1$-parameter family of domain variations. We know that \begin{equation}
\label{eq:nulllagrangian}
\int\limits_{\Mc}|\nabla (u \circ \Psi_t)|^2 \geq
\end{equation} 
$$\geq -\frac{1}{(2m-1)!}\left[ (\om_i)^{2m-1} \wedge (u \circ \Psi_t)^* \Om_I + (\om_j)^{2m-1} \wedge (u \circ \Psi_t)^* \Om_J + (\om_k)^{2m-1} \wedge (u \circ \Psi_t)^* \Om_K\right],$$ by (\ref{eq:triholoinequality}), with equality for $t=0$ since $u$ is triholomorphic. Now

$$ \frac{d}{dt}  (u \circ \Psi_t)^* \Om_I = \frac{d}{dt}  \Psi_t^* u^* \Om_I = \mathcal{L}_X  u^* \Om_I = d(\iota_X  u^* \Om_I),$$
where we used (\ref{eq:strongapproximability}) in the last equality. Using Stokes theorem

$$ \frac{d}{dt} \int\limits_{\Mc} (\om_i)^{2m-1} \wedge (u \circ \Psi_t)^* \Om_I = \int\limits_{\Mc} d((\om_i)^{2m-1}) \wedge(\iota_X  u^* \Om_I).$$
Doing the same for the remaining two terms we get

$$\frac{d}{dt}\int\limits_{\Mc}\left[ (\om_i)^{2m-1} \wedge (u \circ \Psi_t)^* \Om_I + (\om_j)^{2m-1} \wedge (u \circ \Psi_t)^* \Om_J + (\om_k)^{2m-1} \wedge (u \circ \Psi_t)^* \Om_K\right] =$$
$$=\int\limits_{\Mc} d((\om_i)^{2m-1}) \wedge(\iota_X  u^* \Om_I) + d((\om_j)^{2m-1}) \wedge(\iota_X  u^* \Om_J) +d((\om_k)^{2m-1}) \wedge(\iota_X  u^* \Om_K) $$
which is bounded in modulus by

$$(2m-1)! \int\limits_{\Mc} C |X| |\nabla u|^2, \text{ where } C= \|d \om_i \|_\infty+  \|d \om_j \|_\infty+ \|d \om_k \|_\infty.$$
Putting together the latter equation and (\ref{eq:nulllagrangian}) we find

\begin{equation}
\label{eq:almoststationary}
\left.\frac{d}{dt}\right|_{t=0}\int\limits_{\Mc}|\nabla (u \circ \Psi_t)|^2 \geq -C \int\limits_{\Mc} |X| |\nabla u|^2.
\end{equation}
If $\om_i$, $\om_j$, $\om_k$ are closed then $C = 0$ and we have that $u$ is stationary harmonic. If $C \neq 0$ then we have a bound on how much the energy can decrease. In other words (\ref{eq:almoststationary}) gives a notion of almost-stationarity. Remark that by zooming around any point of $\Mc$ at very small scales, we can make $C$ as small as we wish. This notion of almost-stationarity has the same flavour as the notions of almost area minimizing current or semicalibrated current in the setting of Plateau's problem.

\begin{oss}
We do not give the argument to show the weak harmonicity of $u$ for local variations in the target $\Nc \hookrightarrow \R^Q$ when $\Mc$ is hyperK\"ahler. The argument is similar to the above and does not even require assumption (\ref{eq:strongapproximability}), compare also \cite{RT1}. Moreover by dropping the closedness assumption on  $\om_i$, $\om_j$, $\om_k$ we find once again almost stationarity w.r.t. local variations in the target, namely:  for every $\phi \in C^{\infty}_c (\Mc, \R^Q)$ and denoting with $\Pi_{\Nc}$ the nearest point projection from $\R^Q$ onto $\Nc$ (well defined in a tubular neighbourhood of $\Nc$) we have $\left.\frac{d}{dt}\right|_{t=0}\int\limits_{\Mc}|\nabla \left(\Pi_{\Nc} \circ (u + t \phi)\right)|^2 \geq -C \int\limits_{\Mc} |\phi| |\nabla u|^2$.
\end{oss}

\textbf{Monotonicity properties}. We will see next that the Dirichlet energy of a triholomorphic map $u:\Mc \to \Nc$ satisfying (\ref{eq:strongapproximability}) has good almost-monotonicity properties. In the special important case that $\Mc$ is hyperK\"ahler then $u$ is actually a stationary harmonic: in this case the almost-monotonicity formula for the energy ratio of triholomorphic maps follows from the corresponding result for stationary harmonic maps \cite{Price}, \cite{SimonBook}. We present next a direct proof of this monotonicity formula, based directly on the first order PDE (\ref{eq:triholo}) and that has the advantage of extending to the case that $\Mc$ is merely almost hyper-Hermitian.

\begin{Prop}
 \label{Prop:monotonicity}
There exists $r_0>0$ such that for any $x \in \Mc$ and $r<r_0$ the following quantity is monotonically decreasing (weakly) as $r \downarrow 0$:
\begin{equation}
\label{eq:monotonicityofapproxenergy}
\frac{\left(1+(4m-2)r\right)}{r^{4m-2}}\left[\int_{B_r(x)}(\alpha_{1}^{2m-1})\wedge u^* \Om_I + (\alpha_{2}^{2m-1})\wedge u^* \Om_J +(\alpha_{3}^{2m-1})\wedge u^* \Om_K \right].
\end{equation}
This immediately translates into an almost-monotonicity formula for the energy ratio: there exists $r_0>0$ such that for $r\leq r_0$ and $x \in \Mc$
$$\frac{1}{r^{4m-2}}\int_{B_r(x)}|\nabla u|_g^2 = f(r) + O\left(r f(r)\right) \text{ with $f$ a non-decreasing function of $r$}. $$
 
\end{Prop}

\begin{proof}
 
\textbf{Case 1}: Let us consider first the case when $\Mc$ is replaced by the unit ball in $\R^{4m}$, endowed with the flat Euclidean metric and with the standard complex structures $\mathscr{j}_1$, $\mathscr{j}_2$, $\mathscr{j}_3$ and associated K\"ahler forms $\alpha_1$, $\alpha_2$, $\alpha_3$. In other words the domain is the unit ball in $\mathbb{H}^m$, the $m$-dimensional quaternionic vector space. 

We will denote by $x \p_x$ the radial vector field $\sum_{a=1}^{4m} x_a \p_{x_a}$, by $\p_r$ the radial field of unit length $\frac{x \p_x}{|x|}$, for $x=(x_1, ..., x_{4m}) \neq (0, ... ,0)$, and by $dr=d|x|$ the differential of the function $x$ (that is the metric dual to $\p_r$ for the flat metric). Computing explicitly the Lie derivative, for each $\ell \in \{1,2,3\}$, we get $\mathcal{L}_{x \p_x} \alpha_{\ell} = 2 \alpha_{\ell}$ and therefore, distributing the derivative on the wedge product, $$\mathcal{L}_{x \p_x} (\alpha_{\ell}^{2m-1}) = (4m-2) (\alpha_{\ell}^{2m-1}).$$ By the Cartan formula $\mathcal{L} = \iota d + d \iota$ we can see that the $1$-form $\iota_{x \p_x} (\alpha_{\ell}^{2m-1})$ satisfies 
 
$$d(\iota_{x \p_x} (\alpha_{\ell}^{2m-1})) = (4m-2) (\alpha_{\ell}^{2m-1})\,\, \text{ for each } \,\,\ell \in \{1,2,3\}.$$
Denote by $\atanl$ the tangential part of $\alpha_{\ell}^{2m-1}$, that is (for $|x|\neq 0$)

\begin{equation}
\label{eq:tangentialpart}
\frac{1}{(2m-1)!}\atanl = \frac{1}{(2m-1)!}\alpha_{\ell}^{2m-1} - dr \wedge\iota_{\p r} \frac{1}{(2m-1)!}(\alpha_{\ell}^{2m-1}) = \star(\p_r \wedge \mathscr{j}_\ell \p_r) ,
\end{equation}
where $\star$ denotes the Hodge star operation combined with the metric duality vectors/covectors. For the sequel let us remark that (here $|x|\neq 0$)

$$d\left(\frac{\iota_{x \p_x} (\alpha_{\ell}^{2m-1})}{|x|^{4m-2}}\right) = \frac{d (\iota_{x \p_x} (\alpha_{\ell}^{2m-1}))}{|x|^{4m-2}} - \frac{(4m-2) d|x| \wedge \iota_{x \p_x} (\alpha_{\ell}^{2m-1})}{|x|^{4m-1}} =$$ $$=  (4m-2)\frac{  (\alpha_{\ell}^{2m-1})- d|x| \wedge \iota_{ \p_r} (\alpha_{\ell}^{2m-1})}{|x|^{4m-2}} = (4m-2) \frac{\atanl}{|x|^{4m-2}}.$$

We can now start the computations for the monotonicity formula, denoting by $B_R$ and $B_s$ two balls cenetered at the origin with radii respectively $R$ and $s$ (with $R>s$) and keeping in mind that $u^*\Om_I$ and $\alpha_\ell$ are closed forms:

$$\frac{1}{R^{4m-2}} \int_{B_R} (\alpha_1^{2m-1}) \wedge u^*\Om_I - \frac{1}{s^{4m-2}} \int_{B_s} (\alpha_1^{2m-1}) \wedge u^*\Om_I =$$

$$=\int_{\p B_R} \frac{\iota_{x \p_x} (\alpha_{1}^{2m-1})}{(4m-2)R^{4m-2}} \wedge  u^*\Om_I -  \int_{\p B_s} \frac{\iota_{x \p_x} (\alpha_{1}^{2m-1})}{(4m-2)s^{4m-2}} \wedge  u^*\Om_I =$$

$$=\int_{\p B_R} \frac{\iota_{x \p_x} (\alpha_{1}^{2m-1})}{(4m-2)|x|^{4m-2}} \wedge  u^*\Om_I -  \int_{\p B_s} \frac{\iota_{x \p_x}(\alpha_{1}^{2m-1})}{(4m-2)|x|^{4m-2}} \wedge u^*\Om_I =$$
$$=\int_{\p(B_R \setminus B_s)}\frac{\iota_{x \p_x} (\alpha_{1}^{2m-1})}{(4m-2)|x|^{4m-2}} \wedge u^*\Om_I =$$

\begin{equation}
\label{eq:mono1}
=\int_{B_R \setminus B_s}d\left(\frac{\iota_{x \p_x} (\alpha_{1}^{2m-1})}{(4m-2)|x|^{4m-2}}\right)  \wedge u^*\Om_I =  \int_{B_R \setminus B_s}\frac{ \atan[1]}{|x|^{4m-2}} \wedge u^*\Om_I .
\end{equation}

In the last step we used (\ref{eq:strongapproximability}). Observe now, in view of (\ref{eq:tangentialpart}), that 

$$\frac{1}{(2m-1)!}(\atan[1] \wedge  u^*\Om_I)(d\text{vol}^{4m}) =(u^*\Om_I) (\p_r \wedge \mathscr{j}_1 \p_r)=$$

\begin{equation}
\label{eq:mono2}
 = \Om_I  (du(\p_r) \wedge du(\mathscr{j}_1 \p_r)) = \langle I du(\p_r) , du(\mathscr{j}_1 \p_r \rangle_{h|_{u(x)}}= -\langle du(\p_r) , I du(\mathscr{j}_1 \p_r \rangle_{h|_{u(x)}},
\end{equation}
where $h|_{u(x)}$ denotes the metric in the target at the point $u(x)$.

Writing the analogues of (\ref{eq:mono1}) and (\ref{eq:mono2}) for the remaining two terms, i.e. those involving the structures $(\mathscr{j}_2, J)$ and $(\mathscr{j}_3, K)$, and adding up the three parts to get the Dirichlet energy, as indicated in (\ref{eq:DirEnergyTriholo}), we get

\begin{equation}
 \label{eq:mono3}
 \frac{1}{R^{4m-2}}\int\limits_{B_R}|\nabla u|^2 - \frac{1}{s^{4m-2}}\int\limits_{B_s}|\nabla u|^2 = \!\!\!\int\limits_{B_R \setminus B_s} \!\!\!\frac{\langle du(\p_r) , I du(\mathscr{j}_1 \p_r) + J du(\mathscr{j}_2 \p_r )+K du(\mathscr{j}_3 \p_r) \rangle)}{|x|^{4m-2}}  .
\end{equation}
Using the triholomorphic equation $du =    I du \mathscr{j}_1  + J du \mathscr{j}_2  +K du \mathscr{j}_3 $ in the r.h.s. of (\ref{eq:mono3}) we finally conclude

\begin{equation}
 \label{eq:monotonocityformula}
 \frac{1}{R^{4m-2}}\int\limits_{B_R}|\nabla u|^2 - \frac{1}{s^{4m-2}}\int\limits_{B_s}|\nabla u|^2 = \int\limits_{B_R \setminus B_s} \frac{|du(\p_r)|^2}{|x|^{4m-2}}  .
\end{equation}

This is indeed the usual monotonicity formula for stationary harmonic maps having as a domain an Euclidean ball \cite{SimonBook}. Of course there was nothing special in the choice made of centering the balls $B_R$ and $B_s$ at the origin, we can write the same formula for arbitrary centers, by replacing the radial vector field $x \p_{x}$ emanating from the origin with one emanating from the new center. We remark that the classical method for obtaining the monotonicity formula relies on the stationarity of the map with respect to radial variations in the domain, see  \cite{SimonBook}. We have replaced this by the Lie derivative computation with respect to  $x \p_x$, in order to exploit the first oder information (\ref{eq:triholo}). The proof just given is in the style of the proof of the monotonicity formula for calibrated cycles \cite{HL}. 

\medskip

\textbf{Case 2}. In the following we will modify this argument to obtain an almost monotonicity formula for $\Mc$ almost hyper-Hermitian. Let $x_0 \in \Mc$ and consider a geodesic ball around $x_0$, taking normal coordinates so that we work (up to a dilation) in the unit ball of $\R^{4m}$ centered at the origin. By suitably choosing coordinates we can make sure that $\om_i(0)$, $\om_j(0)$, $\om_k(0)$ coincide with $\alpha_1$, $\alpha_2$, $\alpha_3$. In this way $\alpha_1$ is the parallel extension of $\om_i(0)$ to the unit ball (parallel with respect to the flat metric), and similarly $\alpha_2$ and $\alpha_3$ are the parallel extension respectively of $\om_j(0)$, $\om_k(0)$. Since $|\om_i - \alpha_1|(x) \leq \|\nabla \om_i\|_\infty |x|$ we can assume, up to dilating enough, that on the unit ball $B_1(0)$ we have
$$|\om_i - \alpha_1|(x) \leq \epsilon |x|, \,\, |\om_j - \alpha_2|(x) \leq \epsilon |x|, \,\, |\om_k - \alpha_3|(x) \leq \epsilon |x|$$
for some small $\epsilon>0$ and similarly
$$|i - \mathscr{j}_1|(x) \leq \epsilon |x|, \,\, |j - \mathscr{j}_2|(x) \leq \epsilon |x|, \,\, |k - \mathscr{j}_3|(x) \leq \epsilon |x|$$while for the metric, by the choice of normal coordinates, $$|g - g_0|(x) \leq \epsilon |x|^2, $$ where $g_0$ is the Euclidean metric. In other words we are dealing with a small perturbation of the situation that we treated earlier. We choose $\epsilon$ so that $\epsilon<\frac{1}{10 C}$ for the unversal constant $C$ that will appear in (\ref{eq:lowerboundtang}).

\medskip

The first observation is that, since the Dirichlet energy is given by (\ref{eq:DirEnergyTriholo}), in view of the smallness assumptions just made, the quantity

\begin{equation}
 \label{eq:approxenergydensity}
(\alpha_1)^{2m-1} \wedge u^* \Om_I + (\alpha_2)^{2m-1} \wedge u^* \Om_J + (\alpha_3)^{2m-1} \wedge u^* \Om_K
\end{equation}
is a good approximation of the energy density $|\nabla u|_g^2(x)$, i.e. their difference is bounded in modulus (pointwise a.e.) by $\tilde{C}\epsilon|x||\nabla u|_g^2(x)$ for some dimensional constant $\tilde{C}$. Similarly the quantity

\begin{equation}
 \label{eq:approxenergy}
\frac{-1}{R^{4m-2}} \int_{B_R} (\alpha_1)^{2m-1} \wedge u^* \Om_I + (\alpha_2)^{2m-1} \wedge u^* \Om_J + (\alpha_3)^{2m-1} \wedge u^* \Om_K
\end{equation}
is a good approximation of the energy ratio $\frac{1}{R^{4m-2}} \int_{B_R} |\nabla u|_g^2$ , i.e. their difference is bounded in modulus by $\frac{\tilde{C} \epsilon R}{R^{4m-2}} \int_{B_R} |\nabla u|_g^2$ (a higher order term in $R$). So the energy ratio and (\ref{eq:approxenergy}) have the same asymptotic behaviour, therefore we will aim for an almost monotonicity formula for (\ref{eq:approxenergy}). This will be achieved through a modification of the argument given in the flat case above.

A key observation for this modification is to observe that, in the present situation, although the following expression, i.e. the r.h.s. of (\ref{eq:mono1}),

\begin{equation}
\label{eq:approxrhs}
 \atan[1] \wedge  u^*\Om_I +\atan[2] \wedge  u^*\Om_J + \atan[1] \wedge  u^*\Om_K 
\end{equation}
is no longer $\geq 0$, we still have a good lower bound, thanks to the fact that we are handling small perturbations of the previous case. Let us see why. As before we have that this expression is 

$$\text{(expression (\ref{eq:approxrhs}))}=\langle du(\p_r) , I du(\mathscr{j}_1 \p_r) + J du(\mathscr{j}_2 \p_r )+K du(\mathscr{j}_3 \p_r) \rangle = $$ $$=\langle du(\p_r) , I du(i \p_r) + J du(j \p_r )+K du(k \p_r) \rangle- \langle du(\p_r) , I du(i-\mathscr{j}_1) \p_r + J du(j-\mathscr{j}_2) \p_r +K du(k-\mathscr{j}_3) \p_r \rangle$$
$$= \langle du(\p_r),du(\p_r)\rangle - \langle du(\p_r) , I du(i-\mathscr{j}_1) \p_r + J du(j-\mathscr{j}_2) \p_r +K du(k-\mathscr{j}_3) \p_r \rangle \geq$$

\begin{equation}
 \label{eq:lowerboundtang}
 \geq |du(\p_r)|^2 - C\epsilon |x||\nabla u|_g^2,
\end{equation}
for some universal constant $C$. This is the desired lower bound\footnote{In view of this lower bound it is natural to attempt to replace expression (\ref{eq:approxrhs}) in the r.h.s. of (\ref{eq:mono3}) with something of the type $\text{'expression (\ref{eq:approxrhs})} + \tilde{K}  |x| \text{expression (\ref{eq:approxenergydensity})'}$, which is non-negative for $\tilde{K}$ large enough. However we want, in addition, that this non-negative quantity should allow an ``integration by parts-  trick'' as in (\ref{eq:mono1}). For this reason we will need to modify this tentative quantity slightly.}.

% Since we would like a positive quantity, as in the r.h.s. of (\ref{eq:mono3}), we observe next that, in view of (\ref{eq:approxenergydensity}) and (\ref{eq:lowerboundtang}),
% 
% $$\text{(expression (\ref{eq:approxrhs}))} + K \epsilon |x| \text{(expression (\ref{eq:approxenergydensity}))} $$
% is non-negative for $K$ large enough. This non-negative quantity is what we wish to use as ``r.h.s. of a monotonicity formula''. However we want, in addition, that this non-negative quantity should allow an ``integration by parts- trick'' as in (\ref{eq:mono1}). The first term, i.e. (\ref{eq:approxrhs}), does that indeed, whilst the second does not. For this reason we will need to modify it slightly. 
In order to complete the proof we remark that

$$d\left(\frac{\iota_{x \p_x} (\alpha_{\ell}^{2m-1})}{|x|^{4m-3}}\right) = \frac{d (\iota_{x \p_x} (\alpha_{\ell}^{2m-1}))}{|x|^{4m-3}} - \frac{(4m-3) d|x| \wedge \iota_{x \p_x}(\alpha_{\ell}^{2m-1})}{|x|^{4m-2}} =$$ $$=  \frac{ (4m-2) (\alpha_{\ell}^{2m-1})}{|x|^{4m-3}}- \frac{(4m-3) d|x| \wedge \iota_{ \p_r} \alpha_{\ell}}{|x|^{4m-3}} = \frac{ \alpha_{\ell}^{2m-1}+(4m-3)\atanl}{|x|^{4m-3}}=$$

\begin{equation}
 \label{eq:variationsecondterm}
 = |x|\left(\frac{ \alpha_{\ell}^{2m-1}}{|x|^{4m-2}}+\frac{(4m-3)\atanl}{|x|^{4m-2}}\right)
\end{equation}
This suggests\footnote{This is how we modify the tentative quantity from the previous footnote.} that a suitable choice for the ``r.h.s. of a monotonicity formula'' is (here $|x|\neq 0$)

$$d\left(\frac{\iota_{x \p_x} (\alpha_{1}^{2m-1})}{(4m-2)|x|^{4m-2}}+\frac{\iota_{x \p_x} (\alpha_{1}^{2m-1})}{|x|^{4m-3}}\right)\wedge u^* \Om_I + \text{$(\alpha_2,J)$-term} + \text{$(\alpha_3,K)$-term} =$$
$$ =\left( \frac{(1+(4m-3)|x|)\atan[1]}{|x|^{4m-2}} +\frac{ |x|\alpha_{1}^{2m-1}}{|x|^{4m-2}}\right) \wedge u^* \Om_I + \text{$(\alpha_2,J)$-term} + \text{$(\alpha_3,K)$-term} \geq $$
$$\geq \frac{\left(1+(4m-3)|x|\right)\left(|du(\p_r)|^2 - C\epsilon |x||\nabla u|_g^2\right) +|x|(1-\tilde{C}\epsilon|x|)|\nabla u|_g^2}{|x|^{4m-2}}  \geq  $$
$$\geq \frac{\left(1-C\epsilon\right)|x||\nabla u|_g^2-\left((4m-3)C+ \tilde{C}\right)\epsilon|x|^2|\nabla u|_g^2}{|x|^{4m-2}}  \geq 0 $$
for $\epsilon<\frac{1}{10 C}$ and $|x|\leq r_0$ for a small enough $r_0$.

We can now integrate by parts as follows (here $B_s$ and $B_R$ are balls centered at $0$ of radii $s<R<r_0$)

$$0 \leq \!\!\!\int\limits_{B_R \setminus B_s}\!\!\!d\left(\frac{\iota_{x \p_x} (\alpha_{1}^{2m-1})}{(4m-2)|x|^{4m-2}}+\frac{\iota_{x \p_x} (\alpha_{1}^{2m-1})}{|x|^{4m-3}}\right)\wedge u^* \Om_I + \text{$(\alpha_2,J)$-term} + \text{$(\alpha_3,K)$-term} =$$

$$\int\limits_{\p(B_R \setminus B_s)}\!\!\!\left(\frac{\iota_{x \p_x} (\alpha_{1}^{2m-1})}{(4m-2)|x|^{4m-2}}+\frac{\iota_{x \p_x} (\alpha_{1}^{2m-1})}{|x|^{4m-3}}\right)\wedge u^* \Om_I + \text{$(\alpha_2,J)$-term} + \text{$(\alpha_3,K)$-term} = $$

$$=\frac{\left(1+(4m-2)R\right)}{(4m-2)R^{4m-2}}\left[\int\limits_{\p B_R}\iota_{x \p_x} (\alpha_{1}^{2m-1})\wedge u^* \Om_I + \text{$(\alpha_2,J)$-term} + \text{$(\alpha_3,K)$-term}\right] - $$
$$-\frac{\left(1+(4m-2)s\right)}{(4m-2)s^{4m-2}}\left[\int_{\p B_s}\iota_{x \p_x} (\alpha_{1}^{2m-1})\wedge u^* \Om_I + \text{$(\alpha_2,J)$-term} + \text{$(\alpha_3,K)$-term}\right]=$$

$$=\frac{\left(1+(4m-2)R\right)}{R^{4m-2}}\left[\int_{B_R}(\alpha_{1}^{2m-1})\wedge u^* \Om_I + \text{$(\alpha_2,J)$-term} + \text{$(\alpha_3,K)$-term}\right] - $$
$$-\frac{\left(1+(4m-2)s\right)}{s^{4m-2}}\left[\int_{B_R}(\alpha_{1}^{2m-1})\wedge u^* \Om_I + \text{$(\alpha_2,J)$-term} + \text{$(\alpha_3,K)$-term}\right].$$
This provides the desired monotonicity formula (\ref{eq:monotonicityofapproxenergy}), from which the almost monotonicity of Proposition \ref{Prop:monotonicity} follows immediately.

\end{proof}

\textbf{Immediate consequences of almost monotonicity}. The almost monotonicity formula implies that \textit{for every point $x \in \Mc$} there is a well-defined limit of the energy ratio $\frac{1}{r^{4m-2}}\int_{B_r(x)}|\nabla u|_g^2$ as $r \to 0$. This is called energy density of $u$ at $x$ and denoted by $\Theta(u,x)$. A standard argument exploiting the almost monotonicity also gives that $\Theta(u,x)$ is an upper semi-continuous function of $x$.

Moreover there exist $r_0>0$ and $C>0$, depending only on the geometric data, such that whenever $x\in \Mc$ and $r<r_0$ we have $\frac{1}{r^{4m-2}}\int_{B_r(x)}|\nabla u|_g^2 \geq \Theta(u,x) - Cr$.

\section{Analysis of the PDE}

\subsection{Jacobian structure}
\label{jacobianstructure}

\begin{Prop}
\label{Prop:JacobianStructure}
Let $u: \Mc \rightarrow \Nc \subset \R^Q$ be a triholomorphic map in $W^{1,2}(\Mc, \Nc)$ and $\Nc \hookrightarrow \R^Q$. Then $\Delta_g u = (f^1, ... ,f^Q)$, where the $f^j$'s are in the Hardy space $\mathscr{h}^1$ and $\|f\|_{\mathscr{h}^1(\Uc)} \leq C \int\limits_{\Uc}|\nabla u|^2$ for $\Uc \subseteq \Mc$, where $C>0$ is a univeresal constant (i.e. independent of the particular $u$). Here $\Delta_g$ denotes the Laplace-Beltrami operator on $(\Mc,g)$.
% 
% Then $u \in W^{2,1}(\Mc, \R^Q)$ with $$\|u\|_{W^{2.1}(\Mc, \R^Q)} \leq C \int_{\Mc} |\nabla u|^2,$$
% where $C$ depends on $(\Mc, g)$, $\Nc$, $Q$, $i,j,k$ and $I,J,K$.
\end{Prop}

\begin{proof}[\textbf{proof of Proposition \ref{Prop:JacobianStructure}}]
It suffices to prove the result on a chart $\mathcal{U} \subset \Mc$. For any $y \in \Nc \subset \R^Q$ we can extend the almost complex structure $I$ (defined as a linear automorphism on $T_y \Nc$) to a linear automorphism $I_{ext}$ of $T_y \R^Q \cong \R^Q$. For that it suffices to decompose $T_y \R^Q = T_y \Nc \oplus (T_y \Nc)^\bot$ and define the new automorphism $I_{ext}$ to agree with $I$ on $T_y \Nc$ whilst $I_{ext}=0$ on $(T_y \Nc)^\bot$. In the same way we can extend $J$ and $K$ to $J_{ext}$ and $K_{ext}$ respectively. Introducing local coordinates $x^\ell$ on $\Mc$ for $\ell \in \{1, ... 4m\}$ and standard coordinates $y^\alpha$ on $\R^Q$ for $\alpha \in \{1, ... Q\}$. \footnote{We will be using greek indexes on the target and latin indexes on the domain, avoiding in the latter the use of the indexes $i,j,k$ since these letters denote the almost complex structures on $\Mc$.}  With these coordinates in mind the almost complex structures on $\Mc$ and the extended ones in the target will be expressed through coefficients as follows (we write it for $j$ and $J_{ext}$ only, the same notation will be used for $i,k,I_{ext}, K_{ext}$)

$$j\left(\frac{\p}{\p x^\ell}\right) = j_\ell^s \frac{\p}{\p x^s} \, \text{ for } \ell, s \in \{1, ...4m\} $$
\begin{equation}
\label{eq:Jcoefficients}
J_{ext}\left(\frac{\p}{\p y^\alpha}\right) = J_\alpha^\beta \frac{\p}{\p y^\beta} \, \text{ for } \alpha, \beta \in \{1, ...Q\},
\end{equation}
with implicit summation over repeated indexes. Remark that the extended automorphisms are not almost complex structures in $T_y \R^Q$ as they are degenerate in the directions normal to $T\Nc$.

Since on the r.h.s. of (\ref{eq:triholo}) we have $I,J,K$ acting on vectors that are tangent to $\Nc$ we can substitute $I_{ext}$, $J_{ext}$, $K_{ext}$. Therefore the map $u$, as a map $u=(u^1, ..., u^Q)$ into $\R^Q$, satisfies the equation $$du = I_{ext} du \, i + J_{ext} du \, j + K_{ext} du \, k.$$ We will now rewrite this equation in the chosen coordinates. Using (\ref{eq:Jcoefficients}) we can rewrite

$$J_{ext} du \, j \left(\frac{\p}{\p x^\ell} \right)=J_\alpha^\beta j_\ell^s \frac{\p u^\alpha}{\p x^s} \frac{\p}{\p y^\beta} $$
and from here we have the triholomorphic map equation in the form

\begin{equation}
 \label{eq:triholoincoords}
 \frac{\p u^\beta}{\p x^\ell} = I_\alpha^\beta i_\ell^s \frac{\p u^\alpha}{\p x^s} + J_\alpha^\beta j_\ell^s \frac{\p u^\alpha}{\p x^s} + K_\alpha^\beta k_\ell^s \frac{\p u^\alpha}{\p x^s}.
\end{equation}

We will denote, as it is customary, by $g_{ab}$ the metric tensor $g$ in local coordinates and by $g^{ab}$ its inverse. The Beltrami-Laplace operator $div(\nabla)$ on $(\Mc, g)$ is expressed in local coordinates by (always summing over repeated indexes)

$$\Delta_g = \frac{1}{\sqrt{|g|}} \frac{\p}{\p x^a} \left( \sqrt{|g|} g^{a \ell} \frac{\p}{\p x^\ell}\right),$$
where $\sqrt{|g|} dx^1 \wedge ... \wedge dx^{4m}$ is the volume form on  $\Mc$. Applying this operator to $u^\beta$ (for a fixed $\beta \in \{1, ... Q\}$) and using (\ref{eq:triholoincoords}) we find

$$\Delta_g u^\beta= \frac{1}{\sqrt{g}} \frac{\p}{\p x^a} \left( \sqrt{g} g^{a \ell}  I_\alpha^\beta i_\ell^s \frac{\p u^\alpha}{\p x^s} +\sqrt{g} g^{a \ell} J_\alpha^\beta j_\ell^s \frac{\p u^\alpha}{\p x^s} + \sqrt{g} g^{a \ell} K_\alpha^\beta k_\ell^s \frac{\p u^\alpha}{\p x^s}\right) =$$

\begin{equation}
\label{eq:withoutskewsymmetry}
 = \frac{1}{\sqrt{|g|}} \frac{\p}{\p x^a} \left( \sqrt{|g|} g^{a \ell}  j_\ell^s  J_\alpha^\beta \right)\frac{\p u^\alpha}{\p x^s} + g^{a \ell}  j_\ell^s  J_\alpha^\beta \frac{\p^2 u^\alpha}{\p x^a\p x^s} +
\end{equation}
$$+\frac{1}{\sqrt{|g|}} \frac{\p}{\p x^a} \left( \sqrt{|g|} g^{a \ell}  i_\ell^s  I_\alpha^\beta \right)\frac{\p u^\alpha}{\p x^s} + g^{a \ell}  i_\ell^s  I_\alpha^\beta \frac{\p^2 u^\alpha}{\p x^a\p x^s} + $$ 
$$+ \frac{1}{\sqrt{|g|}} \frac{\p}{\p x^a} \left( \sqrt{|g|} g^{a \ell}  k_\ell^s  K_\alpha^\beta \right)\frac{\p u^\alpha}{\p x^s} + g^{a \ell}  k_\ell^s  K_\alpha^\beta \frac{\p^2 u^\alpha}{\p x^a\p x^s}.$$
The summation over $\ell=1,..., 4m$ in the last expression only involves products of the type $g^{a \ell}  i_\ell^s$, $g^{a \ell}  j_\ell^s$, $g^{a \ell}  k_\ell^s$. Recalling the definition of $\om_i=g^{-1}i$ (the non-degenerate two-form uniquely associated to the metric $g$ and the almost complex structure $i$) we can see that for each $(a,s)$ the term $g^{a \ell}  i_\ell^s$ is the coefficient of $\om_i$ corresponding to $dx^a \otimes dx^s$ when we express $\om_i$ in the local coordinates, i.e. $g^{a \ell}  i_\ell^s = \om_i\left(\frac{\p}{\p x^a} \otimes \frac{\p}{\p x^s} \right)$. The skew-symmetry of $\om_i$ implies that $g^{a \ell}  i_\ell^s = - g^{s \ell}  i_\ell^a$. The same goes for the sum of products $g^{a \ell}  j_\ell^s$, $g^{a \ell}  k_\ell^s$. 

Using the skew-symmetry w.r.t. $(a,s)$ of $g^{a \ell}  i_\ell^s$, $g^{a \ell}  j_\ell^s$, $g^{a \ell}  k_\ell^s$ and the symmetry in $(a,s)$ of $\frac{\p^2 u^\alpha}{\p x^a\p x^s}$ we can see that, summing over all $(a,s)$, the second, fourth and sixth terms of (\ref{eq:withoutskewsymmetry}) vanish.
Using the skew symmetry  of $g^{a \ell}  i_\ell^s$, $g^{a \ell}  j_\ell^s$, $g^{a \ell}  k_\ell^s$ again we can rewrite the first, third and fifth terms of of (\ref{eq:withoutskewsymmetry}) summing on $(a,s)$ with $s < a$ only. These observations lead to

\begin{equation}
 \label{eq:withskewsymmetry}
 \Delta_g u^\beta=  \frac{1}{\sqrt{|g|}} \sum_{s<a} \left( \frac{\p}{\p x^a} \left( \sqrt{|g|} g^{a \ell}  i_\ell^s  I_\alpha^\beta \right)\frac{\p u^\alpha}{\p x^s} -  \frac{\p}{\p x^s} \left( \sqrt{|g|} g^{a \ell}  i_\ell^s  I_\alpha^\beta \right)\frac{\p u^\alpha}{\p x^a} \right)+
\end{equation}

$$ +\frac{1}{\sqrt{|g|}} \sum_{s<a} \left( \frac{\p}{\p x^a} \left( \sqrt{|g|} g^{a \ell}  j_\ell^s  J_\alpha^\beta \right)\frac{\p u^\alpha}{\p x^s} -  \frac{\p}{\p x^s} \left( \sqrt{|g|} g^{a \ell}  j_\ell^s  J_\alpha^\beta \right)\frac{\p u^\alpha}{\p x^a} \right)+$$

$$+\frac{1}{\sqrt{|g|}} \sum_{s<a} \left( \frac{\p}{\p x^a} \left( \sqrt{|g|} g^{a \ell}  k_\ell^s  K_\alpha^\beta \right)\frac{\p u^\alpha}{\p x^s} -  \frac{\p}{\p x^s} \left( \sqrt{|g|} g^{a \ell}  k_\ell^s  K_\alpha^\beta \right)\frac{\p u^\alpha}{\p x^a} \right).$$
The coefficients $I_\alpha^\beta$, $J_\alpha^\beta$, $K_\alpha^\beta$ are evaluated at $u(x)$. The functions $I_\alpha^\beta$, $J_\alpha^\beta$, $K_\alpha^\beta$ are smooth on $\Nc$ (which is a compact closed submanifold of $\R^Q$), so they can be extended to a tubular neighbourhood first and then to the whole of $\R^Q$ as smooth functions. As $u \in W^{1,2}(\Mc, \R^Q)$, the functions $I_\alpha^\beta \circ u$, $J_\alpha^\beta \circ u$, $K_\alpha^\beta \circ u$, that appear in the r.h.s. of (\ref{eq:withskewsymmetry}), are therefore in $W^{1,2}(\Mc, \R)$. Setting (for each $(a,s)$) $\psi=\sqrt{|g|} g^{a \ell}  i_\ell^s  I_\alpha^\beta$ and $\phi=u^\beta$ we can recognise, in the r.h.s. of (\ref{eq:withskewsymmetry}), a sum of terms of the type $\frac{\p \psi}{\p x^a} \frac{\p \phi}{\p x^s} - \frac{\p \psi}{\p x^s}\frac{\p \phi}{\p x^a} $, for $\psi$ and $\phi$ in $W^{1,2}(\Mc, \R)$. The Jacobian structure of this expression is the crucial assumption to infer, from \cite{CLMS}, that  $\frac{\p \psi}{\p x^a} \frac{\p \phi}{\p x^s} - \frac{\p \psi}{\p x^s}\frac{\p \phi}{\p x^a} $ lies\footnote{One of the versions of the result \cite{CLMS}, that applies to our situation, is the following. We state it in $\R^m$, taking it from \cite{HelW}, Section 4.2. Let $\alpha \in W^{1,2}(\R^m)$ and $\beta$ be a closed (in distributional sense) $(m-1)$-form of $\R^m$ with coefficients in $L^2(\R^m)$; then $d \alpha \wedge \beta = f dx^1 \wedge ... \wedge d x^m$, where $f$ belongs to the Hardy space $\mathcal{H}^1(\R^m)$. From this we see the following: given $\psi$ and $\phi$ in $W^{1,2}(\R^m)$, let us set $\beta= d (\phi dx^3 \wedge .. \wedge d x^m )$ (this is a closed $(m-1)$-form) and $\alpha = \psi$. The explicit expression of $d \alpha \wedge \beta$ is then $\left(\frac{\p \psi}{\p x^1} \frac{\p \phi}{\p x^2} - \frac{\p \psi}{\p x^2}\frac{\p \phi}{\p x^1}\right)dx^1 \wedge ... \wedge dx^m$. We can apply this to our situation for any couple $(a,s)$.} in the Hardy space $\mathscr{h}^1(\Mc)$. Moreover the Hardy norm of $\frac{\p \psi}{\p x^a} \frac{\p \phi}{\p x^s} - \frac{\p \psi}{\p x^s}\frac{\p \phi}{\p x^a} $ is controlled by $\|\nabla \psi\|_{L^{2}} \|\nabla \phi\|_{L^{2}}$.
So we can conclude that (\ref{eq:withskewsymmetry}) is a PDE of the form $\Delta_g u^\beta = f \in \mathscr{h}^1(\Mc)$ with $\|f\|_{\mathscr{h}^1} \leq Const_{(\Mc, \Nc)} E(u)$. 

\end{proof}

A classical result, e.g. \cite{Stein} \cite{RS}, gives that any solution to $\Delta u = f$ with $f$ in the Hardy space $\mathscr{h}^1$ belongs to $W^{2,1}$ and $\|u\|_{W^{2,1}} \leq Const \|f\|_{\mathscr{h}^1}$.  In the next subsection we show that the same holds for our map $u$ by adapting the result for the classical Laplacian to the Laplace-Beltrami operator by a perturbation argument. We will thus establish the $W^{2,1}$-estimate of Theorem \ref{thm:W21estimate}.

% \begin{thm}
%  \label{thm:W21estimate}
% Let $u: \Mc \rightarrow \Nc \subset \R^Q$ be as in Proposition \ref{Prop:JacobianStructure}. Then
% $$u \in W^{2,1}(\Mc) \text{ and } \|u\|_{W^{2,1}(\Mc)} \leq Const_{(\Mc, g)} \left(Const_{(\Nc \hookrightarrow \R^Q)} + E(u)\right),$$
% where $E(u)$ denotes the Dirichlet energy of $u$.
% \end{thm}

\subsection{Hardy space and proof of Theorem \ref{thm:W21estimate}}
\label{LaplaceBeltrami}

Let $\Phi$ be a function in the Schwartz class $\mathcal{S}(\R^{4m})$ of rapidly decaying smooth functions such that $\int \Phi =1$. For $t>0$ denote by $\Phi_t$ the function $\frac{1}{t^{4m}} \Phi(\frac{\cdot}{t})$. The inhomogeneous Hardy space $\mathscr{h}^1$ defined in terms of a maximal function: given $f \in  L^1(\R^{4m})$ the classical Hardy-Littlewood maximal function is $Mf(x)=\sup_B \frac{1}{|B|} \left|\int_B f\right|$, with the supremum taken over all balls centered at $x$. We will consider a modification of this, namely $\sup_{0<t<1} |\Phi_t \ast f|(x)$ (roughly speaking the balls in the classical maximal function cannot get as large as we want - this is a local version of the Hardy-Littlewood maximal function). The inhomogeneous Hardy space $\mathscr{h}^1$  is the following subset of $L^1(\R^{4m})$ 

\begin{equation}
 \mathscr{h}^1=\{f\in L^1(\R^{4m}): \sup_{0<t<1} |\Phi_t \ast f| \in L^1(\R^{4m})\}
\end{equation}
endowed with the norm $\displaystyle \|f\|_{\mathscr{h}^1} = \left\| \sup_{0<t<1} |\Phi_t \ast f| \right\|_{L^1(\R^{4m})}$. The particular choice of $\Phi$ does not affect the resulting space and gives an equivalent norm. 

The PDE (\ref{eq:withskewsymmetry}) for a component of $u$ on a geodesic ball $B_r \subset \Mc$ can be rewritten in normal coordinates in the following form (implicitly summing over $i,j$ and dropping the index $\beta$ indicating the component for the rest of this section in order to keep a lighter notation)

$$\Delta u + (\sqrt{|g|} g^{ij}-\delta^{ij}) \p_{ij} u +\p_i(\sqrt{|g|} g^{ij}) \p_ju = \sqrt{|g|} f, $$
for $f \in \mathscr{h}^1$ (set $f=0$ in the complement of $B_r$). Here $\delta^{ij}$ denotes the Kronecker delta, $\Delta$ is the usual Laplacian and $(\sqrt{|g|} g^{ij}-\delta^{ij}) = O(|x|^2)$ and $\p_i(\sqrt{|g|} g^{ij}) = O(|x|)$. Up to a dilation of the (small enough) geodesic ball $B_r \subset \Mc$ we can assume that we are working on a ball $B_R(0)$ with radius $R$ much larger than $1$. We suitably re-define $u$, $g$ and $f$ replacing them by $u\left(\frac{r}{R} \cdot\right)$, $g\left(\frac{r}{R}\cdot\right)$ and $f\left(\frac{r}{R}\cdot\right)$. By abuse of notation we will however keep the notation $u$, $g$ and $f$ on $B_R$. With this in mind we have $(\sqrt{|g|} g^{ij}-\delta^{ij})$ and $\p_i(\sqrt{|g|} g^{ij})$ much smaller than $1$ on $B_R(0)$. 

Consider a smooth bump function $\chi$ compactly supported in $B_R(0)$ and identically equal to $1$ on $B_{R/2}(0)$, with $|\nabla \chi|$ and $|\nabla^2 \chi|$ much smaller than $1$. Extending $u$ by setting it identically $0$ in the complement of $B_R(0)$, the function $\chi u$ is well-defined on the whole of $\R^{4m}$. Compute

$$\Delta(\chi u) = \Delta \chi \, u + \nabla \chi \nabla u + \chi \Delta u = $$ $$(\Delta \chi) \, u + \nabla \chi \cdot \nabla u - \chi {\mu}^{ij} D_{ij} u - \chi \tau^j D_j u + \chi \sqrt{|g|} f,$$
where ${\mu}^{ij}=\sqrt{|g|} g^{ij}-\delta^{ij}$ and $\tau^j=\p_i(\sqrt{|g|} g^{ij})$ are smooth functions of $x \in \R^{4m}$ and small in modulus. The latter PDE is valid on $\R^{4m}$ thanks to the multiplication by $\chi$. It can be studied as a perturbation of the PDE $\Delta (\chi u) = \chi \sqrt{|g|} f,$ thanks to the fact that the remaining terms are small. Recall that if $v$ solves the PDE $\Delta v = \tilde{f}$ with $\tilde{f} \in \mathscr{h}^1$ we have that $v \in W^{2,1}$ and $\|v\|_{W^{2,1}} \leq C \|\tilde{f}\|_{\mathscr{h}^1}$.

Consider the space $F=\{w \in W^{2,1}(\R^{4m}): \nabla^2 w \in \mathscr{h}^1(\R^{4m})\}$ with the norm $\|w\|_{W^{1,1}} + \|\nabla^2 w\|_{\mathscr{h}^1}$. We will find by a fixed point technique a unique solution in $F$ to the PDE

\begin{equation}
\label{eq:fixedpoint}
\left\{ \begin{array}{ccc}
\Delta v =  (\Delta \chi) \, v + \nabla \chi \cdot \nabla v - \chi {\mu}^{ij} D_{ij} v - \chi \tau^j D_j v + \chi \sqrt{|g|} f \\
v(x) \to 0  \text{ as } |x| \to \infty
        \end{array}. \right.
\end{equation}
For this purpose, recall the Calderon-Zygmund operator $CZ: f \rightarrow \int \frac{1}{|x-y|^{4m}} f(y) dy$. This operator is justified by writing the solution to $\Delta v = f$ on $\R^{4m}$ (the unique solution under the condition ``boundary data $0$ at infinity'') as convolution with the Green function $G(y)=\frac{1}{|y|^{4m-2}}$: differentiating $v$ twice yields that $\nabla^2 \Delta^{-1} f$ is controlled by the convolution of $f$ with the  Calderon-Zygmund kernel on $\R^{4m}$, i.e. $\int \frac{1}{|x-y|^{4m}} f(y) dy$. A deep result, see e.g. \cite{RS}, shows that $CZ$ is a bounded operator from $\mathscr{h}^1$ to $\mathscr{h}^1$.

\medskip

For any $w \in F$ we can find a unique solution $v \in F$ (with good estimates for the norm $\|\cdot\|_F$) to the PDE

\begin{equation}
\label{eq:fixedpointequation}
\Delta v =  (\Delta \chi) \, w + \nabla \chi \cdot \nabla w - \chi {\mu}^{ij} D_{ij} w - \chi \tau^j D_j w + \chi \sqrt{|g|} f,
\end{equation}
with boundary condition $v(x) \to 0  \text{ as } |x| \to \infty$. To see this, let us analyse the r.h.s. of (\ref{eq:fixedpointequation}). The function $D_{ij} w$ belongs to $\mathscr{h}^1$ since $w \in F$ and multiplication with a smooth compactly supported function preserves the property\footnote{Let $a \in \mathscr{h}^1$ and $b \in L^\infty$. Then $\left(\Phi_t \ast (ab)\right)(x) = \int \Phi_t(x-y) a(y) b(y) dy$ and so $|\Phi_t \ast (ab)|(x) \leq \|b\|_\infty |\Phi_t \ast a|(x)$, whence $\|ab\|_{\mathscr{h}^1} \leq \|b\|_\infty \|a\|_{\mathscr{h}^1}$.} of being in $\mathscr{h}^1$. The same goes for $\chi \sqrt{|g|} f$. The functions $w$ and $\nabla w$ are respectively in $W^{2 ,1}$ and in $W^{1,1}$, that are subspaces of $\mathscr{h}^1$. Thus all the terms on the r.h.s. are in $\mathscr{h}^1$ and so is their sum. Writing equation (\ref{eq:fixedpointequation}) as $\Delta v = \tilde{f}$ we have seen that $\tilde{f} \in \mathscr{h}^1$ and the unique solution is given by $\int \frac{1}{|x-y|^{4m-2}} \tilde{f}(y) dy$. In addition to $\|v\|_{W^{2,1}} \leq C \|\tilde{f}\|_{\mathscr{h}^1}$ we also have 

$$\|\nabla^2  v\|_{\mathscr{h}^1} \leq \|CZ(\tilde{f})\|_{\mathscr{h}^1} \leq Const \, \|\tilde{f}\|_{\mathscr{h}^1}.$$
Let us prove that the map $w \rightarrow Tw =v$ just constructed is a contraction from $F$ into itself. Given $w_1, w_2 \in F$, the difference $v_1 - v_2 := Tw_1 - T w_2$ solves

$$\Delta (v_1 - v_2) =  (\Delta \chi) \, (w_1 - w_2) + \nabla \chi \cdot \nabla (w_1 - w_2) - \chi {\mu}^{ij} D_{ij} (w_1 - w_2) - \chi \tau^j D_j (w_1 - w_2)  $$
and by the estimates recalled above we can see that $\|Tw_1 - T w_2\|_F $ is bounded by the Hardy-norm of the r.h.s. Recalling the smallness assumptions on the moduli of $\nabla^2 \chi$, $\nabla\chi$, $\tau^j$ and $\mu^{ij}$ we find that 

$$\|Tw_1 - T w_2\|_F \leq \epsilon \|w_1 - w_2\|_F$$
for some $\epsilon <1$, i.e. $T$ is a contraction. The Banach-Caccioppoli fixed point theorem yields a unique fixed point of $T$ in $F$, in other words a solution $v \in F$ of the PDE (\ref{eq:fixedpoint}). We also have $\|v\|_F \leq \epsilon \|v\|_F + const \|\chi \sqrt{|g|} f\|_{\mathscr{h}^1}$, so absorbing the first term on the r.h.s.

$$\|v\|_F \leq  Const \|\chi \sqrt{|g|} f\|_{\mathscr{h}^1}.$$

\medskip

Equation (\ref{eq:fixedpoint}) reduces on $B_{R/2}$ to
$$\Delta v + (\sqrt{|g|} g^{ij}-\delta^{ij}) \p_{ij} v +\p_i(\sqrt{|g|} g^{ij}) \p_jv = \sqrt{|g|} f , $$
since $\chi$ is identically $1$ there. So $v$ and $u$ satisfy the same PDE on $B_{R/2}$. Their difference thus solves

$$\p_i\left(\sqrt{|g|}g^{ij} \p _j (u-v)\right) =0 \text{ on } B_{R/2}. $$
So by standard $L^p$-theory for divergence type elliptic operators $u-v$ is smooth and we have $\|u-v\|_{W^{2,p}(B_{R/4})} \leq Const_R \|u-v\|_{L^p(B_{R/4})}$. Recall that $u$ is $L^\infty$ to start with, since it was a coordinate function for a map taking values in the compact submanifold $\Nc \subset \R^N$. Choosing $p=\frac{4m}{4m-2}>1$ (by the Sobolev embedding $v \in W^{2,1}(\R^{4m})$ which embeds in $L^p$ and $\|v\|_{L^p} \leq const \|v\|_{ W^{2,1}}$) and recalling that the $W^{2,1}$-norm is controlled from above by the $F$-norm, we have $$\|u-v\|_{W^{2,p}(B_{R/4})} \leq Const_R \left(\|u\|_{L^p(B_{R/4})} + \|v\|_{L^p(B_{R/4})}\right) \leq Const_R \left(\|u\|_{L^p(B_{R/4})} + \|v\|_{W^{2,1}}\right) $$ $$ \leq Const_R \left(\|u\|_{L^p(B_{R/4})} + \left\|\chi \sqrt{|g|} f\right\|_{\mathscr{h}^1}\right)\leq Const_R \left(\|u\|_\infty + \|f\|_{\mathscr{h}^1}\right),$$
from which it follows that

$$\|u\|_{W^{2,1}(B_{R/4})} \leq  \|v\|_F + \|u-v\|_{W^{2,p}(B_{R/4})} \leq Const_R \left(\|u\|_\infty + \|f\|_{\mathscr{h}^1}\right).$$
% (actually we have the estimate in the $F$-norm but it is not needed) 
The geodesic balls $B_r$ in the beginning must be chosen so that $B_{r/4}$ form a finite cover of $\Mc$ and $g^{ij}-\delta^{ij}$ is uniformly small in the $C^2$-norm (this is possible since $\Mc$ is compact and we have smooth data on it). Then since all the radii $r$ are uniformly bounded away from $0$ and $\infty$ the dilations that send $B_r$ to $B_R$ produce a fixed multiplicative constant when the final estimate is rewritten in the original ball $B_r$. We can then add up these estimates as the geodesic ball ranges over the finite cover.

\medskip

Repeating argument above for every component of the map $u: \Mc \to (\Nc \hookrightarrow \R^Q)$ we have the desired result

$$u \in W^{2,1}(\Mc) \text{ and } \|u\|_{W^{2,1}(\Mc)} \leq Const_{(\Mc, g)} \left(Const_{(\Nc \hookrightarrow \R^Q)} + E(u)\right),$$
where $E(u)$ denotes the Dirichlet energy of $u$. Remark that we have also shown that $D^{ij}u \in \mathscr{h}^1$ with norm bounded by $\left(Const_{(\Nc \hookrightarrow \R^Q)} + E(u)\right)$, therefore it also holds $\Delta u \in \mathscr{h}^1$ (with $\Delta$ rather than $\Delta_g$) with the same estimate on $\| \Delta u\|_{\mathscr{h}^1}$.

\subsection{Consequences: $\eps$-regularity.}
\label{epsreg}

The PDE $\Delta u \in \mathscr{h}^1$ leads to Morrey decay estimates showing an $\eps$-regularity result for a triholomorphic map $u \in W^{1,2}(\Mc, \Nc)$ satisfying (\ref{eq:strongapproximability}), i.e. we get the same results that are well-known in the case of stationary harmonic maps, although our triholomorphic maps are only almost-stationary. The proof exploits the duality $bmo - \mathscr{h}^1$, just as in the fundamental works on stationary harmonic maps \cite{Evans} and \cite{Beth}. Due to the importance of the result for the subsequent development of the paper and for the sake of clarity and completeness we give in this section the explicit argument for the key lemma in this regularity result. 

The space of ``local bounded mean oscillation'' is defined as 

\begin{equation}
 \label{eq:defbmo}
bmo=\left\{f \in L^1_{\text{loc}}: \|f\|_{bmo}= \sup \limits_{|Q|\leq1}\frac{1}{|Q|} \int\limits_{Q}|f(x)-f_Q|+\sup \limits_{|Q|\geq1}\frac{1}{|Q|} \int\limits_{Q}|f(x)|<\infty \right\},
\end{equation}
where $f_Q$ denotes the average on a cube $Q$, $f_Q=\frac{1}{|Q|}\int\limits_{Q} f(x)$. The space $bmo$ is the dual of $\mathscr{h}^1$, see \cite[2.1.2]{RS}  and \cite[2.3.5]{Triebel}  (\cite{FefSt} for the homogeneous spaces $\mathcal{H}^1$ and $BMO$) with duality pairing 

$$(f,g)=\int f g \text{ for } f\in bmo \text{ and } g\in \mathscr{h}^1, \,\,\,\text{ and }\left|\int f g \right| \leq \|f\|_{bmo} \|g\|_{\mathscr{h}^1}.$$

We can then obtain Morrey's Dirichlet growth lemma under a small energy assumption.

\begin{lem}
\label{lem:epsregularitydecay}
There exists $\eps_0>0$, $\tau>0$ and $r_0>0$ depending only on the geometric data on $\Mc$ and $\Nc$ such that whenever $u \in W^{1,2}(\Mc, \Nc)$ satisfies a.e. the triholomorphic equation (\ref{eq:triholo}) and the local strong approximability condition and there exists $B_r(x)$ such that $r\leq r_0$ and $\frac{1}{r^{4m-2}}\int_{B_r(x)}|\nabla u|^2 < \eps_0$, then it holds

$$\frac{1}{(\tau r)^{4m-2}}\int_{B_{\tau r}(x)}|\nabla u|^2 <\frac{1}{2} \frac{1}{r^{4m-2}}\int_{B_r(x)}|\nabla u|^2 .$$
\end{lem}

\begin{proof}
This argument is due to \cite{Evans}, where $\mathbb{S}^N$-valued stationary harmonic maps were studied. The argument goes though in our situation as well, with little modifications, thanks to the similar PDE structure and to the almost monotonicity of the energy ratio. Regarding the latter, remark that (by the compactness of $\Mc$) there exist $R_0>0$ and $\tilde{C}>0$ such that, for any $x \in \Mc$ and whenever $r_1<r_2 \leq R_0$, we have $\frac{1}{(r_1)^{4m-2}}\int\limits_{B_{r_1}(x)}|\nabla u|^2 \leq \frac{(1+\tilde{C}r_2)}{(r_2)^{4m-2}}\int\limits_{B_{r_2}(x)}|\nabla u|^2$. This will be needed in the final stages of the proof.

We can assume that $\tau < \frac{1}{2}$ since we need the result for some small $\tau$. We will argue by contradiction and thus consider a sequence 
$$u_q: B_{r_q}(x_q) \to \Nc \subset \R^Q \text{ with } \frac{1}{r_q^{4m-2}}\int\limits_{B_{r_q}(x_q)}|\nabla u_q|^2 = \lambda_q^2 \to 0 \text{ and } r_q \to 0 \text{ as } q \to \infty\,\, (\lambda_q \neq 0)$$ $$\text{ and } \frac{1}{(\tau r_q)^{4m-2}}\int\limits_{B_{\tau r_q}(x_q)}|\nabla u_q|^2 > \frac{1}{2}\frac{1}{(r_q)^{4m-2}}\int\limits_{B_{r_q}(x_q)}|\nabla u_q|^2 \text{ for all } q.$$
We begin by rescaling each map: denote by $a_q:=(u_q)_{x_q,r_q}:=\frac{1}{|B_{r_q}(x_q)|}\int_{B_{r_q}(x_q)}u_q$ and define

$$v_q: B_{1}(0) \to \R^Q, \,\, v_q(y)=\frac{u_q(x_q + r_q y) - (u_q)_{x_q,r_q}}{\lambda_q}.$$
Then 

\begin{equation}
\label{eq:rescaledv_q}
\int_{B_1(0)}|\nabla v_q|^2 =1 \text{ and } \int_{B_{\tau}(0)}|\nabla v_q|^2 >\frac{1}{2} 
\end{equation}
for every $q$. Remark that the $v_q$'s have zero average, so they have a unform $W^{1,2}$-bound. We can then extract a subsequence (not relabeled) such that 

$$v_q \to v \text{ strongly in } L^2(B_1(0)) \text{ and } \nabla v_q \to \nabla v \text{ weakly in } L^2(B_1(0)).$$
The map $v$ is then harmonic: this can be seen as follows. Consider the PDE (\ref{eq:withskewsymmetry}) for $u_q$, substitute $u_q(x)=\lambda_q v_q(\frac{x}{r_q}-x_q) + a_q$ and rewrite the PDE (changing variable $y=\frac{x}{r_q}-x_q$) on $B_1(0)$ for $v_q$: the l.h.s. becomes

$$\frac{\lambda_q}{r_q^2}\Delta v_q^\beta.$$
Let us see how the r.h.s. transforms. To avoid a heavy notation, remark that the terms that we must analyse are of the form $\p\left(\mu(x)I(u)\right)\p u$ for smooth functions $\mu$ and $I$ (depending only on the geometric structures). Moreover each of these terms comes with an associated term with which it forms a jacobian. After the substitution we get for each such term
$$\frac{\lambda_q}{r_q}(\p \mu)(r_q(y+x_q)) I(\lambda_q v_q(y)) \p v_q (y) + \frac{\lambda_q^2}{r_q^2} \mu(r_q(y+x_q))(\p I)(\lambda_q v_q(y)) \p v_q (y) \p v_q(y),$$
and setting $\tilde{\mu}_q(y)=\mu(r_q y+x_q)$ and multiplying by $\frac{r_q^2}{\lambda_q}$ the latter becomes

$$(\p_y \tilde{\mu}_q)(y) I(\lambda_q v_q(y)) \p_y v_q (y) + \tilde{\mu}_q(y)(\p_x I)(\lambda_q v_q(y)) \p_y v_q (y) \p_y v_q(y),$$
and therefore the PDE for $v_q$ has the schematic form
$$\Delta_g v_q = (\nabla\tilde{\mu}_q) I(\lambda_q v_q) \nabla v_q+\lambda_q \tilde{\mu}_q\left\{\nabla  v_q, \nabla v_q\right\}_I + J\text{-terms} + K\text{-terms}, $$
where $\left\{\nabla  v_q, \nabla v_q\right\}_I$ denotes a jacobian term and the other functions appearing are smooth, bounded and depend on $g$, $i$ $\nabla i$,  $I$, $\nabla I$ (similarly for the remaining $J$ and $K$-terms). Let us look at the first term on the r.h.s. Remark that $\nabla\tilde{\mu}_q \to 0$ in $C^2$ as $q \to \infty$ (recall that $r_q \to 0$) whilst $I(\lambda_q v_q) \nabla v_q$ are bounded in $L^2$ independently of $q$. Let us look at the second term on the r.h.s. Thanks to the jacobian structure it is in the Hardy space $\mathscr{h}^1$ with $\mathscr{h}^1$-norm (and thus $L^1$-norm as well) bounded by $\lambda_q K \int_{B_1(0)}|\nabla v_q|^2 = \lambda_q K$ for a fixed constant $K$. In the weak form the PDE reads, for every $\varphi \in C^\infty_c (B_1(0))$

% $$\Delta_g u^\beta=  \frac{1}{\sqrt{|g|}} \sum_{s<a} \left( \frac{\p}{\p x^a} \left( \sqrt{|g|} g^{a \ell}  i_\ell^s  I_\alpha^\beta \right)\frac{\p u^\alpha}{\p x^s} -  \frac{\p}{\p x^s} \left( \sqrt{|g|} g^{a \ell}  i_\ell^s  I_\alpha^\beta \right)\frac{\p u^\alpha}{\p x^a} \right)+$$
% 

$$\int_{B_1(0)}\langle \nabla v_q, \nabla \varphi \rangle  = \int_{B_1(0)} \varphi O(r_q) \nabla v_q+\lambda_q \int_{B_1(0)} \varphi \tilde{\mu}_q\left\{\nabla  v_q, \nabla v_q\right\}_I+J\text{-terms} + K\text{-terms}$$
and so $\left|\int_{B_1(0)}\langle \nabla v_q, \nabla \varphi \rangle \right| \leq (\lambda_q + r_q)K \|\varphi\|_\infty \to 0$. By the weak convergence $\nabla v_q \to \nabla v$ we conclude

$$\int_{B_1(0)}\langle \nabla v, \nabla \varphi \rangle = 0 \text{ for all } \varphi \in C^\infty_c (B_1(0)),$$
i.e. $v$ is harmonic on $B_1(0)$. In particular by Bochner's formula $|\nabla v|^2$ is subharmonic. Therefore the average of $|\nabla v|^2$ on $B_\tau(0)$ is controlled from above by the average of $|\nabla v|^2$ on $B_{1}(0)$, i.e.

\begin{equation}
\label{eq:estimate_v}
\frac{1}{\tau^{4m-2}}\int_{B_{\tau}(0)}|\nabla v|^2 \leq   \tau^2  \int_{B_{1}(0)}|\nabla v|^2 = \tau^2 \leq \frac{1}{4}.
\end{equation}
The proof will be concluded by showing that the convergence $\nabla v_q \to \nabla v$ is actually strong in $L^2(B_{\frac{1}{2}}(0))$, thereby reaching a contradiction with (\ref{eq:rescaledv_q}). This is the more delicate part of the proof and here the duality $bmo - \mathscr{h}^1$ plays a fundamental role.

Consider $\zeta \in C^\infty_c (B_1(0))$ with $0\leq \zeta \leq 1$ identically $1$ on $B_{\frac{1}{2}}(0)$ and identically $0$ on the complement of $B_{\frac{5}{8}}(0)$. The previous weak equations for $v_q$ and $v$ give, for all test functions $\varphi$,

\begin{equation}
 \label{eq:v_q-v}
\int_{B_1(0)}\langle \nabla v_q - \nabla v, \nabla \varphi \rangle  = \lambda_q \int_{B_1(0)} \varphi\left( \left\{\nabla  v_q, \nabla v_q\right\}_I+\left\{\nabla  v_q, \nabla v_q\right\}_J+\left\{\nabla  v_q, \nabla v_q\right\}_K\right)+
\end{equation}
$$+ \int_{B_1(0)} O(r_q) \, \varphi \nabla v_q. $$
We will now set $\varphi = \zeta^2(v_k -v)$ in order\footnote{This test function is in $W^{1,2}_0(B_1(0)) \cap L^{\infty}(B_1(0)) $ rather than in $C^\infty_c(B_1(0)) $: such test functions can be seen to be valid by a density argument.} to obtain an $L^2$-estimate on $\nabla(v_k -v)$ as follows: the l.h.s. of (\ref{eq:v_q-v}) becomes
$$\int_{B_1(0)}\zeta^2 \left|\nabla (v_k -  v)\right|^2+ 2  \int_{B_1(0)} \zeta (v_k -v) \langle \nabla(v_k -v), \nabla \zeta \rangle \geq $$

$$\geq \int_{B_{\frac{1}{2}}(0)}\left|\nabla (v_q -  v)\right|^2+ o(1), $$
by the strong $L^2$ convergence of $v_q-v$ to $0$. It will thus suffice to prove that the right hand side of (\ref{eq:v_q-v}) goes to $0$, in other words that 

$$\int_{B_1(0)} \zeta^2(v_q -v) \left(\left\{\nabla  v_q, \nabla v_q\right\}_I+\left\{\nabla  v_q, \nabla v_q\right\}_J+\left\{\nabla  v_q, \nabla v_q\right\}_K\right) $$
is bounded independently of $q$. We will regard this integral in the spirit of the $bmo - \mathscr{h}^1$ duality. Recall the Hardy space control (independently of $q$)$$\|\left\{\nabla  v_q, \nabla v_q\right\}_I\|_{\mathscr{h}^1(B_1(0))}+\|\left\{\nabla  v_q, \nabla v_q\right\}_J\|_{\mathscr{h}^1(B_1(0))}+\|\left\{\nabla  v_q, \nabla v_q\right\}_K\|_{\mathscr{h}^1(B_1(0))}\leq K;$$ we will thus aim for a uniform $bmo$-control of $\zeta^2(v_q -v)$. For $z \in B_{\frac{7}{8}}(0)$ and $r\leq \frac{1}{8}$ we have by the Poincar\'{e} and H\"older inequalities

\begin{equation}
\label{eq:bmoMorrey}
\frac{1}{|B_r(z)|} \int\limits_{B_r(z)} |v_q - (v_q)_{z,r}| \leq \frac{C r}{r^{4m}} \int\limits_{B_r(z)} |\nabla v_q| \leq \frac{C}{r^{4m-1}} \left(\int\limits_{B_r(z)} |\nabla v_q|^2 \right)^{\frac{1}{2}} r^{2m} =  \left(\frac{C}{r^{4m-2}} \int\limits_{B_r(z)} |\nabla v_q|^2 \right)^{\frac{1}{2}}.
\end{equation}
The Morrey-Campanato norm that appeared on the r.h.s. is a natural quantity for stationary harmonic maps, it is invariant under dilations of the domain and we know that it has good monotonicity properties. Indeed, going back to $u_q$ (from the definition of $v_q$) we have

$$\frac{1}{r^{4m-2}} \int\limits_{B_r(z)} |\nabla v_q|^2= \frac{1}{\lambda_q^2}\frac{1}{(r_q r)^{4m-2}} \int\limits_{B_{r_q r}(x_q + r_q z)} |\nabla u_q|^2 \leq$$
using the almost monotonicity, and then the inclusion $B_{\frac{1}{8}r_q}(x_q + r_q z) \subset B_{r_q}(x_q)$,
$$\leq \frac{1}{\lambda_q^2}\frac{8^{4m-2}(1+\tilde{C}r_q)}{(r_q)^{4m-2}} \int\limits_{B_{\frac{1}{8}r_q}(x_q + r_q z)} |\nabla u_q|^2 \leq \frac{1}{\lambda_q^2}\frac{8^{4m-2}(1+\tilde{C}r_q)}{(r_q)^{4m-2}} \int\limits_{B_{r_q}(x_q)} |\nabla u_q|^2=$$
$$ = 8^{4m-2}(1+\tilde{C}r_q) \leq \tilde{C}_{m,g}.$$
Recalling (\ref{eq:bmoMorrey}) we can see that $v_q$ have equibounded $bmo$-seminorms on $B_{\frac{7}{8}}(0)$ (i.e. the mean oscillation on small balls). In order to complete the control of the $bmo$-norm (\ref{eq:defbmo}) it is enough to recall that the $v_q$'s are bounded in $L^2$ because they have zero average.

Moreover by the John-Nirenberg inequality the $v_q$'s are also equibounded in $L^p(B_{\frac{7}{8}}(0))$ for all $1\leq p <\infty$ (we will need later $p=4m$).
Consider now $\zeta v_q$, we need to control, for $z \in B_{\frac{3}{4}(0)}$ and $r \leq \frac{1}{8}$,

$$\frac{1}{|B_r(z)|} \int\limits_{B_r(z)} |\zeta v_q - (\zeta v_q)_{z,r}| \leq  \frac{1}{|B_r(z)|} \int\limits_{B_r(z)} |\zeta v_q - \zeta (v_q)_{z,r}|+\frac{1}{|B_r(z)|} \int\limits_{B_r(z)} |\zeta (v_q)_{z,r} - (\zeta v_q)_{z,r}|.$$
By the smoothness of $\zeta$ we have $(\zeta v_q)_{z,r}=\zeta(p) (v_q)_{z,r}$ for some $p \in B_r(z)$, and therefore the last term in the previous inequality is bounded by (here $C$ denotes a constant depending only on $\zeta$) $C r \frac{1}{|B_r(z)|} \int\limits_{B_r(z)} | v_q|.$ The first term on the r.h.s. on the other hand is controlled by $ \frac{C}{|B_r(z)|} \int\limits_{B_r(z)} | v_q - (v_q)_{z,r}|$. Putting together and using H\"older inequality

$$\frac{1}{|B_r(z)|} \int\limits_{B_r(z)} |\zeta v_q - (\zeta v_q)_{z,r}| \leq C\tilde{C}_{m,g} + \frac{C}{r^{4m-1}} \left(\int\limits_{B_r(z)} | v_q|^{4m}\right)^{\frac{1}{4m}} r^{4m-1} $$
and these are equibounded in $q$ by the $L^{4m}$-estimate above. The estimate of the $bmo$-seminorm for $z \in \R^{4m} \setminus B_{\frac{3}{4}(0)}$ is trivial since $\zeta$ vanishes on the complement of $B_{\frac{5}{8}(0)}$. We have thus found that the r.h.s. of (\ref{eq:v_q-v}) goes to $0$, i.e. $\nabla v_q \to \nabla v$ strongly in $L^2(B_{\frac{1}{2}}(0)$, thereby providing a contradicion between (\ref{eq:rescaledv_q}) and (\ref{eq:estimate_v}).
\end{proof}
 
Iterating the previous result at smaller and smaller scales and following the standard De Giorgi/Morrey's scheme one shows that whenever the energy density at $x$ is below $\eps_0$ the decay $\frac{1}{r^{4m-2}}\int\limits_{B_r(x)}|\nabla u|^2 \leq C r^\gamma$ holds for some $0<\gamma<1$ , from which the uniform $\frac{\gamma}{2}$-H\"older continuity of $u$ in $B_r(x)$ follows (see \cite{Giaquinta}, \cite{Morrey}, \cite{Helein}). The H\"older continuity can then be bootstrapped to $C^\infty$-regularity. This last step can be achieved by standard elliptic theory just as in the case of weakly harmonic maps. Indeed we will now show that, also in the case of almost harmonicity that we are dealing with, the PDE satisfies suitable assumptions. 

The continuity of $u$ in $B_r(x)$ allows us to localize not only in the domain but also in the target, therefore we can write $u: B_r(x) \to \Nc$ using local charts in $B_r(x) \subset \Mc$ and around $u(x) \in \Nc$ (let us use coordinates respectively $x^\ell$ and $y^\beta$). Then taking a local variation of $u$ in the target amounts to comparing the energy of $u$ with that of $u + t \phi$ for $\phi=(\phi^\beta)_{\beta=1}^{4n} \in C^\infty_c(B_r(x))$. We can compute, as in Section \ref{Energy&Monotonicity} (analogously to (\ref{eq:nulllagrangian}) with equality for $t=0$ and the computations that followed) 

\begin{equation}
 \label{eq:weakequationalmost}
\left.\frac{d}{dt}\right|_{t=0} \int_{B_r(x)} |\nabla (u+t\phi)|^2 =
\end{equation}
$$=\int\limits_{B_r(x)} d((\om_i)^{2m-1}) \wedge \left(u^*(\iota_\phi \Om_I\right)) + d((\om_j)^{2m-1}) \wedge\left(u^*(\iota_\phi \Om_J\right)) +d((\om_k)^{2m-1}) \wedge\left(u^*(\iota_\phi \Om_K\right)). $$
The l.h.s. of (\ref{eq:weakequationalmost}) can be rewritten using the expression of the metrics $g$ and $h$ (respectively on $\Mc$ and $\Nc$) in the local coordinates and leads, with standard computations (see e.g. \cite{Helein} or \cite{Jost}), to the expression

$$\int \nabla^g u^\gamma  \nabla^g\eta_\gamma+ g^{ij} \Gamma_{\alpha, \beta}^\gamma \frac{\p u^\alpha}{\p x^i} \frac{\p u^\beta}{\p x^j} \eta_\gamma.$$
With the r.h.s. of (\ref{eq:weakequationalmost}) we are thus introducing a perturbation term in the weakly harmonic map equation. Let $\Om_I$ be expressed in local coordinates as a sum of terms of form $c_{\mu \nu}(y)dy^\mu \wedge dy^\nu$. Then $u^*\left(\iota_\phi \Om_I\right)$ is a sum of terms of the form $c_{\mu \nu}(u)\phi^\mu \wedge du^\nu$, therefore the perturbation term in the weak harmonic map equation is of the type $f(x,u, \nabla u)$ with smooth dependence on the variables and linear dependence in the first derivatives of $u$: it satisfies therefore all the assumption (A1)-(G2) of \cite[Section 6.2]{Jost}. The smoothness of $u$ thus follows. 

\medskip

We can further obtain the following result. When $u$ is as in Lemma \ref{lem:epsregularitydecay} there exist $\eps_0$ and $r_0$ such\footnote{We can of course take the $\eps_0$ and $r_0$ here to be equal to those from Lemma \ref{lem:epsregularitydecay}.} that on each ball $B_r(x)$ with $r<r_0$ and on which $\frac{1}{r^{4m-2}}\int_{B_r(x)}|\nabla u|^2<\eps_0$ we have that $u$ is smooth on $B_r(x)$ and the following gradient estimate holds:

\begin{equation}
 \label{eq:gradientestimate}
\sup\limits_{y \in B_{\frac{r}{2}}(x)} |\nabla u(y)| \leq \frac{C \sqrt{{\eps}_0}}{r}.
\end{equation}
The proof of this estimate requires a blow up argument along the lines of \cite[Theorem 2.2]{Schoen}. 

\medskip

To conclude we have

\begin{Prop}
\label{Prop:epsregularity}
There exists $\eps_0>0$ depending only on the geometric data on $\Mc$ and $\Nc$ such that whenever $u \in W^{1,2}(\Mc, \Nc)$ satisfies (\ref{eq:triholo}) and (\ref{eq:strongapproximability}) then 

$$\Theta(u,x) < {\eps}_0 \Rightarrow  x \text{ is a regular point of $u$},$$
i.e. there exists $r>0$ such that $u$ is smooth in $B_r(x)$ (and actually $\Theta(u,x) =0$).  Moreover (\ref{eq:gradientestimate}) holds.
\end{Prop}
In view of this, the singular set of $u$, i.e. the closed set

$$\text{Sing}_u:= \{x \in \Mc: \nexists r>0 \text{ such that } u \text{ is smooth in } B_r(x)\},$$
can be characterized as follows

\begin{equation}
 \label{eq:defsing}
 \text{Sing}_u = \{x \in \Mc: \Theta(u,x) \geq {\eps}_0\}.
\end{equation}

\begin{Prop}
\label{Prop:singset}
With the same assumptions as in Proposition \ref{Prop:epsregularity} we have $$\text{Sing}_u \text{ is a closed set and }\,\, {\Hc}^{4m-2}(\text{Sing}_u)=0.$$
\end{Prop}
The proof is rather standard and involves a covering argument and the monotonicity formula, see e.g. \cite[3.5]{Helein}. (in our case we use the almost monotonicity instead but no change is needed).

\section{Blow up analysis}
\label{blowupanalysis}
Consider the following compactness issue. Let $\{u_\ell\}_{\ell\in \N}$ be a sequence of $W^{1,2}$ triholomorphic maps from $(\Mc, g)$ into $\Nc \subset \R^Q$, with $d(u_\ell^* \alpha)=0$ whenever $\alpha$ is a closed $2$-form on $\Nc$. Assume a uniform energy bound, i.e. there exists $\Lambda>0$ such that $\int_{\Mc} |\nabla u_\ell|^2 \leq \Lambda$ for all $\ell$. Then, up to the extraction of a subsequence, that we do not relabel, we have $u_\ell \rightharpoonup u$ weakly in $W^{1,2}$, for some $u \in W^{1,2}(\Mc, \Nc)$. What can we say about the limiting map $u$?

Recall the almost monotonicity result.
Consider the set
% \footnote{The slight difference in the definition of $\Sigma$, comparing to \cite{Lin}, is due to the fact that the Dirichlet energy is almost monotone, so either we have to use the monotone quantity from (\ref{eq:monotonicityofapproxenergy}) or, in order to use the Dirichlet energy, we must define the set using the liminf of sets $\cup_{L \in \N} \cap_{\ell\geq L}$.} 
(here $C$ , $\eps_0$ and $r_0$ are as in Proposition \ref{Prop:epsregularity}, Lemma \ref{lem:epsregularitydecay} and Proposition \ref{Prop:monotonicity}, recall that they depend only on the geometric data on $\Mc$ and $\Nc$)

 $$\Sigma:=\bigcap_{0<r<r_0}\bigg\{x \in \Mc \text{ such that} $$

$$\left.\liminf\limits_{\ell \to \infty}\frac{1+(4m-2)r}{r^{4m-2}}\left[\int_{B_r}(\alpha_{1}^{2m-1})\wedge u^* \Om_I +(\alpha_{2}^{2m-1})\wedge u^* \Om_J +(\alpha_{3}^{2m-1})\wedge u^* \Om_K \right]\geq {\eps}_0\right\}=$$
\begin{equation}
 \label{eq:defbubblingset}
 =\bigcap_{0<r<r_0} \left\{ x \in \Mc :\liminf_{\ell\to \infty} \frac{1}{r^{4m-2}}\int\limits_{B_r(x)}|\nabla u_\ell|^2 \geq {\eps}_0 - Cr \right\} .
 \end{equation}
This is called the \textit{blow-up set}, or \textit{bubbling set}: it is the (closed) set where the energy concentrates and 

\textbf{Fact}: away from $\Sigma$ we have locally smooth convergence of $u_\ell$ to $u$.

This can be seen as follows: consider $x \notin \Sigma$, i.e. there exists $0<r<r_0$ such that for a subsequence $u_{\ell_j}$ we have $\frac{1}{r^{4m-2}}\int\limits_{B_r(x)}|\nabla u_{\ell_j}|^2<\eps_0$ and thus $u_{\ell_j}$ are smooth in $B_r(x)$ and satisfy (\ref{eq:gradientestimate}) in $B_{\frac{r}{2}}(x)$, so that we have $C^{1,\alpha}$-convergence of $u_{\ell_j}$ to $u$ in $B_{\frac{r}{2}}(x)$. In particular $u$ is $C^{1,\alpha}$ on $\Mc \setminus \Sigma$. Remark that by the $C^{1,\alpha}$-convergence we can pass the PDE to the limit and $u$ is a triholomorphic smooth map on $B_r(x)$. The local $C^{1,\alpha}$-convergence can then be bootstrapped to $C^k$-convergence for any $k$. Moreover $|\nabla u_\ell|^2 d{\Hc}^{4m}  \rightharpoonup |\nabla u|^2 d{\Hc}^{4m}$ on $\Mc \setminus \Sigma$ as Radon measures. Following a well-known scheme involving a covering argument and the almost monotonicity of the energy ratio (see e.g. \cite[Lemma 1.5]{Lin}) we have 

\textbf{Fact}: $\Sigma$ has locally finite ${\Hc}^{4m-2}$-measure (and locally finite Minkowski content).

\medskip

This is proved by the following local argument. Take $\Sigma \cap \overline{B}_1$, it is compact. Consider $0<\de<\frac{\eps_0}{10 C}$ and take a finite cover of $\Sigma \cap \overline{B}_1$ with balls $B_{\de}(x_j)$ of radius $\de$ and centered at points $x_j \in \Sigma$. By Vitali's covering result we can choose a subfamily of these balls such that $B_{\de}(x_j)$ are disjoint and $B_{3\de}(x_j)$ covers $\Sigma \cap \overline{B}_1$. For all $\ell$ large enough we have that (for every $j$) $\frac{1}{{\de}^{4m-2}}\int_{B_{\de}(x_j)}|\nabla u_\ell|^2 \geq \eps_0 - C \de \geq \frac{\eps_0}{2}$. Therefore 

$$\sum_j {\de}^{4m-2} \leq \frac{2}{{\eps}_0} \int_{\cup_j B_{\de}(x_j)}|\nabla u_\ell|^2 \leq \frac{2}{{\eps}_0} \int_{\Mc}|\nabla u_\ell|^2 \leq \frac{2}{{\eps}_0}  \Lambda,$$
where we used that the $B_{\de}(x_j)$'s are disjoint. Then we have, recalling $\Sigma \cap \overline{B}_1 \subset \cup_j B_{3 \de}(x_j)$, that

$${\Hc}^{4m-2}(\Sigma \cap \overline{B}_1) \leq 3^{4m-2}  \frac{2}{{\eps}_0}  \Lambda.$$
Remark that this result can be refined, in the following way. Since $\cup_j B_{3 \de}(x_j)$ covers $\Sigma \cap \overline{B}_1$, then the open set $\cup_j B_{4 \de}(x_j)$ contains the $\de$-neighbourhood of $\Sigma \cap B_{1}$, i.e. the set $\{y: B_1: \text{dist}(y, \Sigma \cap \overline{B}_{1}) < \delta\}$, therefore this $\de$-neighbourhood has $(4m)$-dimensional measure at most $4^{4m}  \frac{2}{{\eps}_0}  \Lambda \de^2$. In other words the $2$-dimensional \textit{Minkowski content} of $\Sigma \cap \overline{B}_1$ is bounded by $4^{4m}  \frac{2}{{\eps}_0}  \Lambda$.
 
\medskip

By $C^{1,\alpha}$-convergence to $u$ on $\Mc \setminus \Sigma$, which implies the strong $W^{1,2}$-convergence, we also have, for any $2$-form $\alpha$ on $\Nc$, the strong $L^2$-convergence $u_\ell^*(\alpha) \to u^*(\alpha)$ when we restrict to an open set in the complement of $\Sigma$. Therefore $u$ satisfies the strong approximability condition away from $\Sigma$. However this \textit{does not necessarily hold across $\Sigma$}, where the $W^{1,2}$-convergence will generally fail to be strong. We now analyse this aspect further.

Consider the Radon measures $|\nabla u_\ell|^2 d{\Hc}^{4m}$. Without loss of generality we may assume $|\nabla u_\ell|^2 d{\Hc}^{4m}  \rightharpoonup \mu$ weakly as Radon measures. By Fatou's lemma $\mu= |\nabla u|^2 d{\Hc}^{4m}+\nu$ for some non-negative Radon measure $\nu$ on $\Mc$. The previous discussion then gives that $\nu$ is supported inside $\Sigma$ and $\text{Sing}_u$ is a subset of $\Sigma$. On the other hand, whenever $x \in \Sigma$ then we have that for every $0<r<r_0$ there exists $L\in \N$ such that and for every $\ell\geq L$ it holds $\frac{1}{r^{4m-2}}\int\limits_{B_r(x)}|\nabla u_\ell|^2 \geq \frac{{\eps}_0}{2}$ and therefore for a.e. $0<r<r_0$ we must have $\frac{\mu(B_r(x)}{r^{4m-2}}\geq \frac{{\eps}_0}{2}$. If $u$ is smooth at $x$ then $\frac{1}{r^{4m-2}}\int\limits_{B_r(x)}|\nabla u|^2 =O(r^2)$ and therefore for a.e. small $r$ we have $\frac{\nu(B_r(x))}{r^{4m-2}}\geq \frac{{\eps}_0}{4}$, i.e. $x$ is in the support of $\nu$. We conclude therefore 

\textbf{Fact}: $\Sigma = \text{Sing}_u \cup \text{spt}\nu$.

\medskip

Similarly to \cite{Lin} Lemma 1.6 we further have the following facts:

\textbf{(a)} For every $x \in \Mc$ the ratio $\displaystyle \frac{\mu(B_r(x))}{r^{4m-2}}$ is almost monotone in $r$, i.e. for $r\leq r_0$ it is the sum $f(r)+O(rf(r))$ of a non-decreasing function of $r$ and an infinitesimal function. This allows to define the density of $\mu$ for all $x$, $\Theta(\mu,x):=\lim_{r\to 0}\frac{\mu(B_r(x))}{r^{4m-2}}$. This follows by passing to the limit $\int_{B_r(x)}|\nabla u_\ell|^2 \to \mu(B_r(x))$ for a.e. $r<r_0$ and by using the almost monotonicity formula for $u_\ell$.

\textbf{(b)} $ x \in \Sigma \Leftrightarrow \Theta(\mu,x) \geq \eps_0$. Indeed, if $ \Theta(\mu,x) \geq \eps_0$ then for $r<r_0$ we have $\frac{\mu(B_r(x))}{r^{4m-2}} \geq \eps_0 - Cr$ by (a) and therefore $\liminf_{\ell \to \infty}\frac{1}{r^{4m-2}}\int_{B_r(x)}|\nabla u_\ell|^2 \geq \eps_0 - Cr$ by the weak convergence of measures and thus $x \in \Sigma$. On the other hand if $x \in \Sigma$ then for a.e. $r < r_0$ we have $\frac{\mu(B_r(x))}{r^{4m-2}} \geq \eps_0 -Cr$ and sending $r \to 0$ we find that $\Theta(\mu,x) \geq \eps_0$.

\textbf{(c)} For the $W^{1,2}$ function $u$ the set $\left\{x: \limsup_{\rho \to 0} \frac{1}{r^{4m-2}}\int_{B_\rho(x)}|\nabla u|^2>0\right\}=\{x: \Theta_u(x) >0\}$ is a set of zero $\Hc^{4m-2}$-measure, see \cite[2.4.3. and 4.8]{EvansGariepy}. In the complement of it $u$ is approximately continuous (in the sense of Lebesgue points).
% In our case, by Proposition \ref{Prop:epsregularity}, this means that for $\Hc^{4m-2}$-a.e. $x\in \Sigma$ we have that $u$ is continuous. 
Moreover we can pass to the limit, away from $\Sigma$, the second order PDE that we obtained in (\ref{eq:weakequationalmost}) because we have $C^{1,\alpha}$-convergence, i.e. on the complement of $\Sigma$ the limiting map $u$ satisfies the weak formulation of

\begin{equation}
 \label{eq:weakharmonicPDEperturbed}
\Delta_g u + A(\nabla u, \nabla u) = f(x,u,\nabla u) ,
\end{equation}
that is (in a first moment we must take $\phi\in C^\infty_c(\Mc \setminus \Sigma)$)

\begin{equation}
 \label{eq:weakharmonicPDEperturbedweakform}
 \int \nabla u \nabla \varphi+ A(\nabla u, \nabla u) \varphi = \int f(x,u,\nabla u) \varphi\,\,\, \text{ on } \Mc,
\end{equation}
where $A$ denotes the second fundamental form of $\Nc \subset \R^Q$ and $f$ is a smooth function depending only on the geometric structure on $\Mc$ and $\Nc$ and depending linearly on the third variable. Moreover this PDE extends across the $\Hc^{4m-2}$-dimensional set $\Sigma$ by a standard capacity argument, so $u$ satisfies the weak formulation of this second order PDE globally, i.e. we can allow $\phi\in C^\infty_c(\Mc)$ in (\ref{eq:weakharmonicPDEperturbedweakform}). 
At a continuity point of $u$ we can localize it in the target and use the results from \cite{Jost} as we did in (\ref{eq:weakequationalmost}) to infer the smoothness of $u$.

\textbf{Fact}: for $\Hc^{4m-2}$-a.e. $x\in \Sigma$ the map $u$ satisfies $\Theta(u,x)=0$ and it is there approximately continuous. 
% The set $\text{Sing}_u$ has zero $2$-capacity and is contained in in $\Sigma$.

\medskip

The three facts (a)-(b)-(c) imply that the density of $\mu$ at points of $\Sigma$ and the density of $\nu$ w.r.t. the measure $\Hc^{4m-2}$ coincide a.e. on $\Sigma$ (in particular the density of $\nu$ exists) and the value of the density is $\geq \eps_0$. From \cite{Preiss} we deduce

\textbf{Fact}: the set $\Sigma$ is $(4m-2)$-rectifiable\footnote{In \cite{Lin} Lin gives a self-contained argument for the rectifiability, without relying on the deep recifiability result of Preiss.}. The measure $\nu$ can be expressed as $\nu= \Theta(x) \Hc^{4m-2}(x)\res \Sigma$ and $\Theta(x) \geq \eps_0$ for a.e. $x$.

\medskip

The limiting map $u$ solves (\ref{eq:triholo}) and (\ref{eq:weakharmonicPDEperturbedweakform}), however the weak $W^{1,2}$-convergence $u_\ell \rightharpoonup u$ is not enough to infer that $d(u^* \alpha)=0$ whenever $\alpha$ is a closed $2$-form on $\Nc$. This causes a \textit{possible lack of almost-stationarity} of $u$, which is a very important and possibly hard issue to analyse\footnote{This was overlooked in \cite{ChenLi1}.}. Focusing on the case in which $\Mc$ is hyperK\"ahler (i.e. the maps $u_\ell$ are stationary harmonic), all we can say about the limit is that the couple $(u, \Sigma)$ is stationary in the sense of \cite{LiTian}\footnote{Roughly speaking, the first variations of $\Sigma$ as a varifold and of $u$ as a map might fail to vanish but they annihilate each other.}, but a priori neither $\Sigma$ nor $u$ need to be. Determining whether they are is an open question, as well as the regularity properties of both $\Sigma$ and $u$.

\section{Quantization of the energy - proof of Theorem \ref{thm:quantization}}
\label{quantization}

In this section we prove Theorem \ref{thm:quantization}. Consider $u_\ell: \Mc \to \Nc$, a sequence of  $W^{1,2}$ triholomorphic maps satisfying (\ref{eq:strongapproximability}) with a uniform energy bound $\int |\nabla u_\ell|^2 \leq \Lambda<\infty$ for all $\ell \in \N$. By the results in the previous section we have $u_\ell \rightharpoonup u$ weakly in $W^{1,2}$, where $u$ is a triholomorphic map, with strong convergence away from a set $\Sigma$ that is $(4m-2)$-rectifiable and has locally finite ${\Hc}^{4m-2}$-measure. The map $u$ is smooth away from the set $\text{Sing}_u \subset \Sigma$, and is approximately continuous away from a ${\Hc}^{4m-2}$-negligeable set. We moreover have

$$|\nabla u_{\ell}|^2 d{\Hc}^{4m} \rightharpoonup |\nabla u|^2 d{\Hc}^{4m} + \Theta(x)d{\Hc}^{4m-2},$$
with $\Theta \geq \eps_0>0$ a.e. on $\Sigma$. The question that we analyse in this section is the following. Is it true that for ${\Hc}^{4m-2}$-a.e. $x \in \Sigma$ we have 

$$\Theta(x)=\sum_{s=1}^{S_x} E(\phi_s),$$ where each $\phi_s:S^2 \to \Nc$ is a (smooth) non-constant harmonic map? This is known as \textit{quantization of the defect measure}. We will answer the question affirmatively. In the case that $\Mc$ and $\Nc$ are hyperK\"ahler the affirmative answer was already given in \cite{Wang} with a different approach. We will give an alternative proof, covering also the the more general case we are analysing, with $\Mc$ almost hyper-Hermitian. Our proof relies on Lorentz space estimates, as in \cite{LR}, exploiting the structure of the PDE that we analysed in Proposition \ref{Prop:JacobianStructure} and the $W^{2,1}$-regularity result from Theorem \ref{thm:W21estimate}. In view of Section \ref{Holomorphicity properties of smooth blow-up sets}, we will go through the arguments quite comprehensively, omitting some details for which we refer to \cite{LR}. 

\medskip

For the purposes of Theorem \ref{thm:quantization} it is enough to restrict our attention to points $x \in \Sigma$ where the approximate tangent $T_x \Sigma$ exists and such that the limit map $u$ is approximately continuous at $x$: by doing so we only neglect a $\Hc^{4m-2}$-negligeable subset of $\Sigma$, which does not affect the result. 
% The previous dimension reduction is thus achieved, for such points $x$, by simply taking a tangent at $x$ to the couple $(u, \Sigma)$, which is necessarily of the form $(u(x), T_x \Sigma)$, where $u(x)$ denotes the constant map with value $u(x)$.
With this in mind, at any such $x$ and for every $\ell$ we can find a suitable dilation $u_{\ell, \lambda_\ell}$ so that (using a diagonal argument) the dilated maps, that we will denote by $\tilde{u}_\ell$, converge weakly in $W^{1,2}$ to the constant $u(x)$ (the value taken at $x$ by the precise representative of $u$) with blow up set $T_x \Sigma$.\footnote{This is a special case of Federer's reduction argument, that is performed at all points in \cite{Lin}; for a general exposition of this argument we refer to \cite{SimonBook}, Appendix A.}
% 
% the previous abstract construction can be read back in the more concrete picture of the given maps $u_\ell$'s: the maps $\tilde{u}_\ell$ can be produced by a diagonal process after taking (starting from the given $u_\ell$'s) suitable dilations at $x$ ($u_{\ell, \lambda_\ell}$ for suitable $\lambda_\ell$ is obtained by a diagonal argument and $u_{\ell, \lambda_\ell}$ converge to $(u(x), T_x \Sigma)$). 
We thus have $|\nabla \tilde{u}_{\ell}|^2 d{\Hc}^{4m} \rightharpoonup \nu$ and the measure $\nu$ is a multiple (with factor $\Theta(x)$) of the ${\Hc}^{4m-2}$-dimensional Haussdorf measure restricted to the linear $(4m-2)$-dimensional subspace $T_x \Sigma$.

\medskip

In view of the quantization question that we are addressing, we can therefore work straight away on the $\tilde{u}_\ell$ and prove that $\Theta(x)$ (that is now a constant on $T_x\Sigma$) is a sum of energies of non-constant harmonic spheres. For the sake of notational convenience, we will drop the tilde for the rest of the section and assume that $T_x\Sigma$ is $\R^{4m-2} \times (0,0)$ in suitable coordinates. So we assume that we are dealing with a sequence of strongly approximable $W^{1,2}$ triholomorphic maps $u_\ell: B^{4m-2}_1(0) \times B^{2}_1(0)\to \Nc$ with uniformly bounded energies such that $$u_\ell \rightharpoonup \text{cnst} \,\,\,\text{ weakly in }W^{1,2}(B^{4m-2}_1(0) \times B^{2}_1(0)),$$  
where $cnst$ stands for a constant map and the convergence is strong (even $C^{1,\alpha}$) locally in the complement of $B_1^{4m-2}(0) \times (0,0)$; moreover
\be
\label{eq:assumptionforquantization}
|\nabla u_{\ell}|^2 d{\Hc}^{4m} \rightharpoonup \Theta \,\,d{\Hc}^{4m-2} \res \left( B_1^{4m-2}(0) \times (0,0)\right) \,\,\, \text{ for some }  0<{\eps}_0 \leq  \Theta \in \R.
\ee
From the monotonicity applied to $u_\ell$ at the points $0$ and $\xi_1, \xi_2, ..., \xi_{4m-2}$ (where the latter $(4m-2)$-points span $\R^{4m-2} \times (0,0)$) we have, as in \cite[(2.11)]{Lin}, that 

$$\int\limits_{B^{4m-2}_1(0) \times B^{2}_1(0)}\!\!\! \sum\limits_{a=1}^{4m-2}\left| \frac{\p u_\ell}{\p x_a} \right|^2 \to 0 \,\, \text{ as } \ell \to \infty.$$
With the notation $X_1=(x_1, ..., x_{4m-2})$ and $X_2=(x_{4m-1}, x_{4m})$ we then define

$$f_\ell(X_1) =  \int\limits_{B^{2}_1(0)}\!\!\! \sum\limits_{a=1}^{4m-2}\left| \frac{\p u_\ell}{\p x_a} \right|^2(X_1, X_2) d\,X_2$$
and by Fubini's theorem we then have $f_\ell \to 0$ in $L^1(B^{4m-2}_1(0))$. Recall the Hardy-Littlewood maximal function $Mf_\ell$ of $f_\ell$

$$Mf_\ell(Y):= \sup\limits_{0<r<\frac{1}{2}} \frac{1}{r^{4m-2}}\int\limits_{B_r^{4m-2}(Y)} f_\ell(X_1) d\, X_1 \,\, \text{ for } Y \in B_{\frac{1}{2}}^{4m-2}(0)$$
and the weak-$L^1$ estimate  

$${\Hc}^{4m-2}\left(\{Mf_\ell > \lambda\}\right) \leq \frac{C_m}{\lambda}\|f_\ell\|_{L^1}.$$
So for every $\delta>0$ we can find $A_\de \subset  B_{\frac{1}{2}}^{4m-2}(0)$ with ${\Hc}^{4m-2}(A_\de) \geq {\Hc}^{4m-2}(B_{\frac{1}{2}}^{4m-2}(0)) - C_m \de$ and $$\sup\limits_{0<r<\frac{1}{2}} \frac{1}{r^{4m-2}}\int\limits_{B_r^{4m-2}(Y)} f_\ell(X_1) d\, X_1 < \de \,\, \text{ for all } Y \in A_\de \text{ and for all large enough } \ell.$$

Recall further that for each $\ell$ the map $u_\ell$ is smooth away from a ${\Hc}^{4m-2}$-negligeable set\footnote{We are using Proposition \ref{Prop:epsregularity} and the almost continuous representative of a $W^{1,2}$ map as we did in section \ref{blowupanalysis}.}. Taking a countable union we conclude that for ${\Hc}^{4m-2}$-a.e. $X_1 \in B^{4m-2}(0)$ and for all $X_2 \in B^{2}(0)$ all maps $u_\ell$ are smooth at $(X_1, X_2)$. Therefore we have a sequence of points $X_1^\ell \in B_{\frac{1}{2}}^{4m-2}(0)$ such that

\begin{equation}
 \label{eq:choice1}
\text{ each map $u_\ell$ is smooth near $ \{X_1^\ell\} \times B_{1}^{2}(0)$}
\end{equation}
$$\text{and }\sup\limits_{0<r<\frac{1}{2}} \frac{1}{r^{4m-2}}\int\limits_{B_r^{4m-2}(X_1^\ell)} f_\ell(X_1) d\, X_1 \to 0 \, \text{ as }  \ell \to \infty.$$

Now we need to ``track down the first bubble''. For this purpose we show that, for every $\ell$ large enough, we may find $\de_\ell \in \left(0, \frac{1}{2}\right)$ and $X_2^\ell \in B_{\frac{1}{2}}^2(0)$ such that 

\begin{equation}
 \label{eq:choice2}
 \max\limits_{X_2 \in B_{1/2}^{2}(0)}\,\, \frac{1}{{\de}_\ell^{4m-2}} \int\limits_{B_{{\de}_\ell}^{4m-2}(X_1^\ell) \times B_{{\de}_\ell}^{2}(X_2)}|\nabla u_\ell|^2 dx
\end{equation}
$$\text{is achieved at } X_2^\ell \text{ with value }\frac{\eps_0}{8 \cdot 2^{4m}} \text{ and } X_2^\ell  \to (0,0) \text{ for } \ell \to \infty.$$
Indeed, the smoothness of $u_\ell$ - first line of (\ref{eq:choice1})- gives that for each $\ell$ there is a $\de(\ell)>0$ such that for all $X_2 \in B_{\frac{1}{2}}^2(0)$ we have

\be
\label{condI}
\frac{1}{{\de}^{4m-2}} \int\limits_{B_{{\de}}^{4m-2}(X_1^\ell) \times B_{{\de}}^{2}(X_2)}|\nabla u_\ell|^2 dx \leq \frac{1}{2}\frac{\eps_0}{8 \cdot 2^{4m}},
\ee
whenever $\de < \de(\ell)$. On the other hand we may fix any arbitrary $\de >0$ and have, for all $\ell$ large enough, that

\be
\label{condII}
\max\limits_{X_2 \in B_{1/2}^{2}(0)}\,\, \frac{1}{{\de}^{4m-2}} \int\limits_{B_{{\de}}^{4m-2}(X_1^\ell) \times B_{{\de}}^{2}(X_2)}|\nabla u_\ell|^2 dx \geq {\eps}_0,
\ee
for if that were not the case, we would use the $\eps$-regularity result Proposition \ref{Prop:epsregularity} and conclude that $|\nabla u_\ell| \leq \frac{C_m \sqrt{\eps_0}}{\de}$ on $B_{\de/2}^{4m-2}(X_1^\ell) \times B^2_{1/4}(0)$. This would give $C^{1,\alpha}$-convergence of $u_\ell$ on $B_{\de/2}^{4m-2}(X_1^\ell) \times B^2_{1/4}(0)$, contradicting the existence of the blow up set $B_1^{4m-2}(0) \times (0,0)$. Conditions (\ref{condI}) and (\ref{condII}) imply (with the smoothness of $u_\ell$ which gives continuity in $\delta$ of the energy ratio at a point) that for every large enough $\ell$ we can find $\de_\ell$ satisfying (\ref{eq:choice2}) with $\de_\ell \to 0$ as $\ell \to \infty$. Finally let us check that $X_2^\ell \to (0,0)$. If that were not the case, i.e. if there exists $c>0$ s.t. $|X_2^\ell|\geq c$  for all $\ell$ in a subsequence (not relabeled), then we would have (using the almost monotonicity of the energy ratio for balls centered at $(X_1^\ell, X_2^\ell)$) for all $\ell$ large enough so that $c > 4 \de_\ell$

$$\int\limits_{B_{{1}}^{4m-2}(X_1^\ell) \times (B_{{1/2}}^{2}(X_2^\ell) \setminus B_{{c/2}}^{2}(X_2^\ell))}|\nabla u_\ell|^2 dx \geq C_{c,m,{\eps}_0}$$
independently of $\ell$ for all large enough $\ell$. However this contradicts the strong convergence $u_\ell \to cnst$ on $B_{{1}}^{4m-2}(X_1^\ell) \times (B_{{1/2}}^{2}(X_2^\ell) \setminus B_{{c/2}}^{2}(X_2^\ell))$.

\medskip

We consider now the maps $v_\ell$ defined by $v_\ell(y)=u_\ell\left((X_1^\ell, X_2^\ell) + \de_\ell y\right)$. The map $v_\ell$ is defined on $B^{4m-2}_{R_\ell} \times B^{2}_{R_\ell}:=B^{4m-2}_{\frac{1}{2\de_\ell}} \times B^{2}_{\frac{1}{2\de_\ell}}$ and $R_\ell \to \infty$. Then we have 
\begin{equation}
\label{eq:convergenceofthefirstbubble}
v_\ell \to v \text{ in } C^{1,\alpha}_{\text{loc}}(\R^{4m-2}\times \R^2) \text{ as }\ell\to \infty, 
\end{equation}
where $v(y)=v(y_{4m-1}, y_{4m})$ is a smooth nonconstant harmonic map from $\R^2$ into $\Nc$ with finite energy ($v$ does not depend on the remaining $(4m-2)$ variables in view of the second line of (\ref{eq:choice1})). For the proof of the $C^{1,\alpha}_{\text{loc}}$-convergence we refer to \cite[page 818]{Lin}. By Sacks-Uhlenbeck's result \cite{SacksU} we can see $v$ as a smooth harmonic map from $\mathbb{S}^2$ into $\Nc$ (using the conformal change of coordinates induced by stereographic projection from the north pole). This gives the first bubble $\phi_1$.  

Remark that the $C^{1,\alpha}$-convergence is only local on $\R^2$, therefore after the conformal change of variable that lets us see the convergence on $S^2$ we have a sequence of domains exhausting $S^2$ and leaving out a small disk (shrinking as $\ell \to \infty$) around the north pole: the $C^{1,\alpha}$-convergence holds on any domain that leaves out a fixed disk around the north pole. On the remaining neck we need to perform further analysis. First of all we must locate other points where the energy is concentrating (locate other \textit{bubble domains}) and where we can produce more bubbles by repeating the procedure illustrated for the first bubble: ``dilation of disks when needed/reparametrization on $S^2$'' . We will find possibly several but surely finitely many bubbles, since each bubble contributes at least with a fixed positive amount to the energy and we have a global bound. The second task will be to check if there is energy that accumulates in the necks connecting the bubble domains: this will be ruled out and will lead to Theorem \ref{thm:quantization}.

\medskip

By Theorem \ref{thm:W21estimate} we have that $\|\nabla u_\ell\|_{W^{1,1}} \leq C_{\Lambda, \Mc, \Nc}$ independently of $\ell$. Recall now that $W^{1,1}(\R^2)$ embeds continuously in the Lorentz space $L^{2,1}(\R^2)$. We recall the notion of Lorentz spaces

$$L^{2,1}(\R^m) = \left\{f: \|f\|_{L^{2,1}}:=\int_0^\infty \sqrt{\mathcal{L}^m\left(\{|f|\geq t\} \right)}dt <\infty\right\}.$$

$$L^{2,\infty}(\R^m) = \left\{f: \|f\|_{L^{2,\infty}}^2:=\sup_{t>0} t^2 \mathcal{L}^m\left(\{|f|\geq t\} \right) <\infty\right\}.$$
Therefore, by first choosing good slices and then embedding, we get that there exist a constant $\Lambda^*$ and a subset $E_\ell \subset B^{4m-2}_{\frac{1}{2}}$ with measure $\mathcal{L}^{4m-2}(E_\ell) \geq \frac{99}{100}\mathcal{L}^{4m-2}(B^{4m-2}_{\frac{1}{2}})$ such that for all $X_1 \in E_\ell$ we have

$$\|(\nabla u_\ell)(X_1, \cdot)\|_{L^{2,1}(B^2_{\frac{2}{3}}(0))} \leq \Lambda^*.$$
This choice of slices, i.e. values $X_1 \in E_\ell$, should be combined with the choice of $F_\ell \subset B^{4m-2}_{\frac{1}{2}}$ with measure $\mathcal{L}^{4m-2}(F_\ell) \geq \frac{99}{100}\mathcal{L}^{4m-2}(B^{4m-2}_{\frac{1}{2}})$ such that the two conditions in (\ref{eq:choice1}) are satisfied for all $X_1 \in F_\ell$. For the sequel we will work on \textit{``good slices''}, i.e. good values of $X_1$, meaning that $X_1 \in E_\ell \cap F_\ell$. We therefore consider the maps $\overline{u}_\ell(X_2) := u_\ell(X_1^\ell, X_2)$ for a sequence $X_1^\ell$ such that\footnote{In view of Section \ref{Holomorphicity properties of smooth blow-up sets} it should be stressed that we have a huge freedom in choosing this sequence $X^1_\ell$, namely for every $\ell$ we only need to avoid a $\Hc^{4m-2}$-small set, with measure at most $1 \%$ of $\mathcal{L}^{4m-2}(B^{4m-2}_{\frac{1}{2}})$. In Section \ref{Holomorphicity properties of smooth blow-up sets} the $2$-planes $X^1=X^1_\ell$ will be read in the original context (prior to the dimension reduction, i.e. prior to the transition from $u_\ell$ to $\tilde{u}_\ell$).}  for every $\ell$ we have 

$$\text{ each map $u_\ell$ is smooth near $ \{X_1^\ell\} \times B_{1}^{2}(0)$}$$
$$\sup\limits_{0<r<\frac{1}{2}} \frac{1}{r^{4m-2}}\int\limits_{B_r^{4m-2}(X_1^\ell)} f_\ell(X_1) d\, X_1 \to 0 \, \text{ as }  \ell \to \infty.$$
\begin{equation}
\label{eq:choice3}
\|(\nabla u_\ell)(X_1^\ell, \cdot)\|_{L^{2,1}(B^2_{\frac{2}{3}}(0))} \leq \Lambda^*.
\end{equation}
From the blow up analysis we know that
$$\Theta = \lim_{\ell \to \infty} \int\limits_{B^2_{1/2}(0)} |\nabla \overline{u}_\ell|^2$$
and the content of Theorem \ref{thm:quantization} is to show that this is a sum $\sum_{s=1}^{S} E(\phi_s),$ where each $\phi_s:S^2 \to \Nc$ is a (smooth) harmonic map. We have seen before, using (\ref{eq:choice2}) to find suitable points $X^\ell_2$ and suitable dilating factors $\de_\ell$, that $v_\ell(X_2)=\overline{u}_\ell(X_2^\ell + \de_\ell(X_2-X_2^\ell))$ converge strongly in $C^{1,\alpha}_\text{loc}$ to $v$ harmonic from $\R^2$ into $\Nc$, which becomes the first bubble $\phi_1$ after composing with stereographic projection from the north pole.

For any fixed $R$ the strong local convergence on $B_R(0)$ gives $\int_{B_R}|\nabla v_\ell|^2 \to \int_{B_R}|\nabla v|^2$ and by conformal invariance $\int_{B_{\de_\ell R}(X_2^\ell)} |\nabla \overline{u}_\ell|^2 \to \int_{B_R}|\nabla v|^2$. So the energy left (after the first bubble has been accounted for) is 

\begin{equation}
\label{eq:L2sumofspheres}
\lim_{R \to \infty} \lim_{\ell \to \infty} \int_{B^2_{1/2}(0) \setminus B_{\de_\ell R}(X_2^\ell)} |\nabla \overline{u}_\ell|^2
\end{equation}
(here $\nabla$ stands for $D_{X_2}$) and we want to prove that it equals $\sum_{s=2}^{S} E(\phi_s).$

Observe that, for any fixed $c>1$ and for any $R>0$ we have $\int_{B_{c R}\setminus B_{R}}|\nabla v_\ell|^2 \to \int_{B_{c R}\setminus B_{R}}|\nabla v|^2$ and by conformal invariance $\int_{B_{c\de_\ell R}(X_2^\ell)\setminus B_{\de_\ell R}(X_2^\ell)} |\nabla \overline{u}_\ell|^2 \to \int_{B_{c R}\setminus B_{R}}|\nabla v|^2$ as $\ell \to \infty$. Therefore

$$\lim\limits_{R\to \infty }\lim\limits_{\ell \to \infty}\int_{B_{c\de_\ell R}(X_2^\ell)\setminus B_{\de_\ell R}(X_2^\ell)} |\nabla \overline{u}_\ell|^2 = \lim\limits_{R\to \infty }\int_{B_{c R}\setminus B_{R}}|\nabla v|^2 =0. $$
Moreover we have for any $c>1$ (remark that for all $\ell$ large enough the domain ${B_{\frac{1}{2}}(X_2^\ell)\setminus B_{\frac{1}{2c}}(X_2^\ell)}$ stays well away from $\Sigma$, since $X_2^\ell \to (0,0)$):
$$\lim\limits_{\ell \to \infty}\int_{B_{\frac{1}{2}}(X_2^\ell)\setminus B_{\frac{1}{2c}}(X_2^\ell)} |\nabla \overline{u}_\ell|^2=0. \,\,\,\,\footnote{Observe that $$\lim\limits_{\ell \to \infty}\int_{B_{\frac{1}{2}}(X_2^\ell)\setminus B_{\frac{1}{2c}}(X_2^\ell)} |\nabla \overline{u}_\ell|^2 = \lim\limits_{\ell \to \infty}\int_{B_{\frac{1}{2\de_\ell}}(0)\setminus B_{\frac{1}{2c\de_\ell}}(0)} |\nabla v_\ell|^2= 0, $$
i.e. the energies of $v_\ell$, that actually do escape partly to infinity when extra bubbles are produced, must escape ``in a controlled way'', i.e. the domains where these energies are localized (further \textit{bubble domains}) escape to infinity slower than the annuli ${B_{\frac{1}{2\de_\ell}}(0)\setminus B_{\frac{1}{2c\de_\ell}}(0)}$ (for any choice of $c>1$), i.e. with a speed that is of lower-order compared to the speed of growth of the domains $B_{\frac{1}{\de_\ell}}(0)$. We stress, in order to avoid confusion for the reader that is used to Parker's bubble tree construction \cite{Parker}, that here (following \cite{Lin}, \cite{LR}) the bubble map produced by the conformal wrapping of the $2$-disks (domains of the $v_\ell$) onto $S^2$ is constructed, although equivalently in the end, differently than in \cite{Parker}: here we are covering $S^2$ except for a small disk around the north pole (this missing disks shrinks to a point as $\ell\to \infty$) and the extra bubbling domains are escaping towards the north pole as well (at a slower speed); in \cite{Parker} the bubbling domains are confined in the southern hempisphere and it is on this hemisphere that further bubbling analysis is performed. The same goes for the neck domain: here it is an annular domain with one boundary circle  coinciding with the boundary of the missing disk around the north pole (thus going to the north pole as $\ell \to \infty$), while the other boundary circle is fixed when $R$ is fixed (it comes from $\p B_R(0)$) and goes to the north pole as $R \to \infty$. In \cite{Parker}, on the other hand, the neck domain is located close to the south pole. The difference originates from a different choice of the dilating factors $\delta_\ell$.}$$

\medskip

Consider now the change of variables (conformal in $X_2 \in \R^2$) induced by 
\begin{equation}
\label{eq:elongatedneck}
W_\ell(X_1, t, \theta) = u_\ell(X_1, e^{-t}, \theta) \,\,\, \text{ with } \,\, X_2 = (r, \theta) \text{ in polar coordinates.}                                                                    
\end{equation}
 Then $ \int_{B^2_{1/2}(0) \setminus B^2_{\de_\ell R}(X_2^\ell)} |\nabla \overline{u}_\ell|^2$ coincides with

$$E(W_\ell, B_\ell) := \int\limits_{\{X_1^\ell\}\times [\log 2|, \log \de_\ell R|] \times \mathbb{S}^1} |\nabla W_\ell|^2 $$
(where $\nabla$ stands for the standard gradient in $t$ and $\theta$) and we want to show that this is a sum of energies of harmonic spheres arising in the blow-up analysis; we know that this energy cannot accumulate close to the boundaries of the domain $B_\ell$, namely for any $M>0$ (here $c=e^M$ with $c$ used above)

\begin{equation}
 \label{eq:noaccumulationatboundaries}
E(W_\ell, Q(0,M)) := \int_{[\log 2, \log 2 + M] \times \mathbb{S}^1} |\nabla W_\ell|^2 \to 0 \,\, \text{ as }\ell \to \infty 
\end{equation}
$$\lim\limits_{R\to \infty }\lim\limits_{\ell \to \infty}E(W_\ell, Q(\ell,M)) := \lim\limits_{R\to \infty }\lim\limits_{\ell \to \infty}\int_{[|\log \de_\ell R|-M, |\log \de_\ell R|] \times \mathbb{S}^1} |\nabla W_\ell|^2 = 0 . $$

\medskip

Recall that there is a finite number of bubbles only. At this point we face a dichotomy: either there is only one bubble or we can locate another bubble domain (in the new variables). If the second option happens, then we select the next bubble domain by choosing, for each $\ell$, a point $t_\ell \in [\log 2, |\log \de_\ell R|-1]$ such that

$$e^{(4m-2)t} \int\limits_{|X_1|\leq e^{-t}} \int_{t}^{t+1}\int\limits_{\mathbb{S}^1} |\nabla W_\ell|^2 $$
(here $\nabla=\left(e^{-t}D_{X_1}, \frac{\p}{\p t}, \frac{\p}{\p \theta}\right)$) achieves its maximum (for $t \in [\log 2, |\log \de_\ell R|-1]$) at $t=t_\ell$ and this maximum is $\geq \eps_1>0$ for some\footnote{If we can make sure that for any $\eps_1>0$ this energy is below $\eps_1$ then we must be in the case where we have only the first bubble $\phi_1$, as will be clear from the subsequent argument.} fixed $\eps_1$. Note that the map $W_\ell$ is triholomorphic and almost stationary harmonic with respect to $\nabla=(e^{-t}D_{X_1}, \frac{\p}{\p t}, \frac{\p}{\p \theta})$, compare \cite[(2.12)]{LR}\footnote{There is a little misprint, the right hand side of \cite[2.12]{LR} should read $-e^{-2t}\Delta_{X_1} W_i$, where $\Delta_{X_1}$ is the Laplacian in the $X_1$ coordinates.}. Keeping in mind that the cyclinder $[\log 2, |\log \de_\ell R|] \times \mathbb{S}^1$ gets longer and longer as $\ell \to \infty$, we find from estimates (\ref{eq:noaccumulationatboundaries}) that (respectively) $t_\ell \to \infty$ and $|\log \de_\ell R|-t_\ell \to \infty$: in other words, we can center the cyclinders at $t_\ell$ and we will have that this translated cylinders get longer with lengths going to $\infty$ on both sides (with respect to their centre) as $\ell \to \infty$. The centering of the cyclinders corresponds to the change of variable $\tilde{t}=t-t_\ell$. With the extra change of variable $X_1 = e^{-t_\ell} Y_1$ we define $V_\ell(Y_1, \tilde{t}, \theta) = W_\ell(X_1, t-t_\ell, \theta)$. We then have that $V_\ell$ is defined on $B^{4m-2}_2(0) \times [-M_\ell, M_\ell] \times \mathbb{S}^1$ with $M_\ell \to \infty$ and $V_\ell$ is triholomorphic and almost stationary harmonic (compare \cite[(3.3)]{LR}) and

\begin{equation}
 \label{eq:lowerboundV}
\int\limits_{|Y_1|\leq 1} \int_{0}^{1}\int\limits_{\mathbb{S}^1} |\nabla V_\ell|^2 \geq {\eps}_1,
\end{equation}
where $\nabla = \left(e^{-\tilde{t}}D_{Y_1}, \frac{\p}{\p \tilde{t}}, \frac{\p}{\p \theta}\right)$. The Dirichlet energies of $V_\ell$ are uniformly bounded and we may assume that $V_\ell \to V_\infty$ weakly in $W^{1,2}_{\text{loc}}$ (possibly strongly). The situation now is rather similar to the one that we were facing with the maps $v_\ell$ when we produced the first bubble. The map $V_\infty$ is smooth and harmonic on $B^{4m-2}_2 \times (-\infty, \infty) \times \mathbb{S}^1$ and independent of the first $(4m-2)$ variables. We have two cases (compare bottom pictures of Figure \ref{fig:bubbletree} in Section \ref{Holomorphicity properties of smooth blow-up sets}):

\textbf{(i) }If the convergence holds strongly in $W^{1,2}_{\text{loc}}$ then the map $V_\infty$ is not constant thanks to (\ref{eq:lowerboundV}). As such, by conformally changing its domain $(-\infty, \infty) \times \mathbb{S}^1$ to $S^2 \setminus \{\text{the two poles}\}$ we produce the second bubble $\phi_2:S^2 \to \Nc$ (this is a case where further dilation is not needed and we can reparametrise on $S^2$ straight away). Intuitively we have replaced the piece $[0,1]\times \mathbb{S}^1$ with a sphere without two small disks. So instead of the original unique cyclinder (the neck domain we started from) we have a spherical domain (without two disks) where convergence of the $V_\ell$ happens strongly, actually $C^{1,\alpha}$, and two neck domains left from breaking in two the original cyclinder. On each of these two neck domains we start over with the bubbling analysis.

\textbf{(ii)} If the convergence is weak then we proceed as in the case of $\phi_1$: we find a (non-empty, by (\ref{eq:lowerboundV})) blow-up set $\tilde{\Sigma}$, that is of the form $B^{4m-2}_2 \times (\overline{t}, \overline{\theta})$ and proceed to suitable dilations around concentration point of the energy (for a sequence of points converging to $(\overline{t}, \overline{\theta})$). We thus replace a small cap on the cyclinder (around $(\overline{t}, \overline{\theta})$, this is the new bubble domain) with a new $S^2 \setminus \{\text{a small cap}\}$ and attach them by a cyclinder (a new neck domain). The spherical part becomes the domain of $\phi_2$. On the new neck domain we start over with the bubbling analysis. Moreover, on the remaining part of the previous cyclinder (from which we removed a small disk) we proceed to the identification of further bubbling domains, if any are present.

\medskip

This procedure will stop at some point because only a finite number of bubbles can appear, since each bubble contributes with at least a fixed amount of energy and this cannot become infinite ($\Theta$ is finite). Recalling (\ref{eq:L2sumofspheres}), now that all the $\phi_s$ have been identified, the next task is to show that the Dirichlet energies on the (finitely many) \textit{neck domains} go to $0$ as $\ell \to \infty$. To fix strategy and terminology (since topologically the necks are always annuli no matter how many times we perform a conformal change into a cylinder) we will look for an $L^{2,\infty}$-estimate for $\nabla \overline{u}_\ell$ in each neck domain (thinking of it as an annulus on a sphere, as in the original picture obtained when we produced $\phi_1$) by looking for an $L^\infty$-estimate in the \textit{elongated neck}, i.e. the long cylinder obtained by the conformal change of variable (\ref{eq:elongatedneck}). We conclude the section by explaining why we look for such an (apparently) different estimate and how we can get it.

The Dirichlet energy is conformally invariant, so it does not matter how many times we conformally change a neck, therefore we can assume that the elongated neck we are looking at admits no further bubble domains. The aim is to conlcude that $|\nabla \overline{u}_\ell|$ goes to $0$ in the $L^{2,\infty}$-sense in each neck, then we will conclude that the $L^2$-norm must go to $0$ as well by interpolation with the uniform $L^{2,1}$-bound. The $L^{2,\infty}$-estimate follows from a control of the type

\begin{equation}
\label{eq:L2inftyestimate}
|X_2 - X_2^\ell| | \nabla_{X_2} \overline{u}_\ell(X_2)| \leq \sqrt{\eps}  \text{ on each neck domain} 
\end{equation}
for $\eps$ as small as we like as long as we choose $\ell$ and $R$ large enough. Indeed (\ref{eq:L2inftyestimate}) implies that whenever $t>0$ and $|\nabla_{X_2} \overline{u}_\ell(X_2)|\geq t$ then we must have $|X_2 - X_2^\ell|\leq \frac{\sqrt{\eps}}{t}$. This forces the set $\{X_2: |\nabla_{X_2} \overline{u}_\ell(X_2)|\geq t\}$ to be contained in the ball of radius $\frac{\sqrt{\eps}}{t}$ centered at $X_2^\ell$. In particular the $L^{2,\infty}$-norm of $|\nabla \overline{u}_\ell|$ is below $\eps$ by definition. Now recall from (\ref{eq:elongatedneck}) that

$$r=e^{-t}, \,\,\, -\log r=t,\,\, \frac{-dr}{r} = dt, \,\,\,r \frac{\p}{\p r} = \frac{\p}{\p t}$$
and so $|X_2 - X_2^\ell|\nabla_{X_2} \overline{u}$ is nothing but the derivative of $W_\ell$ after the conformal change of variables. In other words (\ref{eq:L2inftyestimate}) amounts to a pointwise $L^\infty$ estimate on $\nabla W_\ell$ in the elongated neck. Remark now that the conformal change of variables makes $W_\ell$ almost stationary harmonic, so in every fixed size domain contained in a neck domain we have that the energy is as small as we wish in the $L^2$ sense (otherwise there would be another bubble domain there, as explained earlier) and by the interior gradient estimate we can conclude a pointwise $L^\infty$-estimate (as small as we wish) on $\nabla W_\ell$. Once we achieve that $|\nabla \overline{u}_\ell|$ goes to $0$ in the $L^{2,\infty}$-sense in the neck domains, using the uniform estimate on the $L^{2,1}$-norms from (\ref{eq:choice3}) we conclude by interpolation\footnote{One way to see this is by exploiting the $L^{2,\infty}-L^{2,1}$ duality, indeed for any $v$ such that $\nabla v \in L^{2,1} \subset L^2 \subset L^{2,\infty}$ we have $\int |\nabla v|^2 = \int \nabla v \cdot \nabla v \leq \|\nabla v\|_{L^{2,1}}\|\nabla v\|_{L^{2,\infty}}$.} that the $L^2$ norm in the neck domains goes to $0$, i.e. in (\ref{eq:L2sumofspheres}) we only find energies of bubbles. We refer to \cite{LR} for the technical adjustments of these ideas.

\medskip

We conlude the section with the following

\begin{Prop}
\label{Prop:complexstructureofbubble}
With the assumptions of Theorem \ref{thm:localholomorphicitysmooth}, for $\Hc^{4m-2}$-a.e. $x \in \Sigma$ there exists $(a_x, b_x, c_x)$ with $a_x^2 + b_x^2+c_x^2=1$ s.t. all the bubbles $\phi_s:S^2 \to \Nc$ at $x$ (using the notation in (\ref{eq:sumofbubbles})) induce holomorphic $S^2$'s in $\Nc$ with respect to the complex structure $-(a_x I + b_x J + c_x K)$. Moreover the approximate tangent plane at $x$ to $\Sigma$ is holomorphic with respect to $a_x i + b_x j + c_x k$.
\end{Prop}

This is shown in \cite{Wang} Lemma 2.2; we want to make transparent the important underlying idea and stress an aspect that will be useful to keep in mind in view of the arguments in Section \ref{Holomorphicity properties of smooth blow-up sets}.
Let $x$ be a ``good point'' on $\Sigma$, i.e. $T_x \Sigma$ exists and the limiting map $u$ is continuous at $x$. The creation of a bubble at $x$ is essentially due to the fact that the Dirichlet energy is accumulating at $x$: more precisely, the components $\nabla_{X_2} u_\ell$ are producing this concentration (here $X^2$ denotes the two directions normal to $T_x \Sigma$).  A suitable reparametrization in the normal directions $X_2$ provides the bubbles, whilst in the directions tangential to $\Sigma$ the components $\nabla_{X_1} u_\ell$ are going to $0$ in $L^2$. This causes the bubbles to depend only on the normal coordinates $X_2$. The equation (\ref{eq:triholo}) passes therefore to the limit to each bubble with a degeneration in the $X_1$-directions, i.e. only the two derivatives in the directions $X_2$ are non-zero. It is easy to check that (i) any two directions identify a complex $2$-dimensional plane with respect to a complex structure $a_x i + b_x j + c_x k$ for a certain choice of $(a_x, b_x, c_x) \in S^2$ (ii) the $(4m-2)$-plane $T_x \Sigma$ is holomorphic for the same structure $a_x i + b_x j + c_x k$ that makes the normal $2$-plane holomorphic. 
A short computation shows that whenever we annihilate two directions in (\ref{eq:triholo}) we find that the map becomes holomorphic with respect to $((a_x i + b_x j + c_x k), -((a_x I + b_x J + c_x K))$, hence the result. The normal $2$-plane can then be conformally transformed into an $S^2 = \CP^1$, as in the standard definition of bubble. 
\begin{oss}
\label{oss:holo&antiholo}
It is worthwile noting the following: it cannot happen that at a fixed point $x$ we produce some bubbles that are holomorphic with respect to $(a_x i+ b_x j+ c_x k), -((a_x I + b_x J + c_x K))$ and some others that are holomorphic with respect to $(a_x i+ b_x j+ c_x k), ((a_x I + b_x J + c_x K))$. With our choice of the PDE (\ref{eq:triholo}) we are allowing, say, $(i,-I)$-holomorphic bubbles to appear but it can be checked immediately that if a map is triholomorphic as in (\ref{eq:triholo}) and it is in addition $(i,I)$-holomorphic, then it must have rank $4$ at least, whilst a bubble has rank $2$. This feature is not typical of the general bubbling issue for stationary harmonic maps, where it could happen indeed that at a certain point we have some bubbles that are holomorphic and some others that are anti-holomorphic.
\end{oss}

\section{Holomorphicity properties of blow-up sets - proof of Theorem \ref{thm:localholomorphicitysmooth}}
\label{Holomorphicity properties of smooth blow-up sets}

Under the asumption that $u$ is smooth on $B^{4m}$ we know that $\Sigma$ coincides with the support of the defect measure $\nu$ in $B^{4m}$, and we are assuming that this support is non-empty. We will prove the statement of Theorem \ref{thm:localholomorphicitysmooth} when $\mathcal{L}$ is a Lipschitzian graph over some chosen $(4m-2)$-plane. Let us see how the proof of the general case follows from this. For $\mathcal{L}$ as in the assumptions of Theorem \ref{thm:localholomorphicitysmooth} we observe that for every point on $\mathcal{L}$ we have a neighbourhood in which $\mathcal{L}$ is a Lipschitzian graph. Consider the set 

$$\mathcal{D}=\left\{x \in \Sigma \cap B^{4m}: \exists \text{  an open neighbourhood $A_x$ of $x$ and $ai+bj+ck$ with $(a,b,c)\in S^2$ } \right.$$ $$\left. \text{ s.t. $\Sigma \cap A_x = \mathcal{L} \cap A_x$ and this is smooth and pseudoholomorphic w.r.t. $ai+bj+ck$}\right\}.$$
The set $\mathcal{D}$ is open in $\mathcal{L}$ (also observe that the points in $\mathcal{D}$ are of density $1$ for $\Sigma$). Let us show that $\mathcal{D}$ is closed in $\mathcal{L}$. Since $\Sigma$ is closed, we have that the closure in $\mathcal{L}$ of $\mathcal{D}$, denoted $\overline{\mathcal{D}}$, is contained in $\Sigma$. If there exists $y \in \overline{\mathcal{D}} \setminus \mathcal{D}$ we choose an open ball $B'$ centered at $y$ such that $\mathcal{L} \cap B'$ is a Lipschitz graph and remark that $\Sigma \cap B'$ has strictly positive measure (it contains points of density $1$). Then applying the result again, we obtain that in a neighbourhood of $y$ the set $\Sigma$ coincides with $\mathcal{L}$ and it is a smooth pseudoholomorphic submanifold, hence $y \in \mathcal{D}$. Having now established that $\Sigma \cap B^{4m}$ is open and closed in $\mathcal{L}$, by the connectedness of $\mathcal{L}$ we get that $\Sigma \cap B^{4m}=\mathcal{L} \cap B^{4m}$ and it is a smooth submanifold. Moreover locally it is pseudo holomorphic for a fixed almost complex structure, which gives a well defined map from $\mathcal{L}$ into $S^2_{IJK}$ and this map is locally constant: the connectedness of $\mathcal{L}$ yields that it is globally constant, completing the proof.

\medskip 

In the rest of the section we prove Theorem \ref{thm:localholomorphicitysmooth} under the further assumption that $\mathcal{L}$ is a Lipschitizian graph over a $(4m-2)$ dimensional plane. 
Recall that by the rectifiability of $\Sigma$ proved in Section \ref{blowupanalysis} we have that for $\Hc^{4m-2}$-a.e. $x \in \Sigma \cap B^{4m}$ we have that $x$ has density $1$ in $\Sigma$ and that there exists an approximate tangent $T_x \Sigma$, that agrees with the tangent to $\mathcal{L}$ at $x$. Moreover $\mathcal{L} \setminus \Sigma$ is open in $\mathcal{L}$. We obtain Theorem \ref{thm:quantization} for all these (almost all) points on $\Sigma \cap B^{4m}$: when we used the dimension reduction in the beginning of Section \ref{quantization} we had to ensure that at the point under consideration the limit map $u$ was continuous in average and the blow-up set possessed an approximate tangent; the second condition is met now everywhere on $\Sigma \cap B^{4m}$. We thus have, from Proposition \ref{Prop:complexstructureofbubble}, a well-defined map from a.e. $\Sigma \cap B^{4m}$ into the $S^2$ representing the complex structures $aI + bJ + cK$ on the hyperK\"ahler $\Nc$: it associates to a.e. $x \in \Sigma \cap B^{4m}$ the complex structure for which the bubbles $\phi_s$ at $x$ are holomorphic. 

Consider two points $P$ and $Q$ on $\Sigma \cap B^{4m}$ where the approximate tangent to $\Sigma$ exists and agrees with the tangent to $\mathcal{L}$, obtained by the differentiability of $\mathcal{L}$ (so at these points, that only exclude a null set, we have the quantization result). The strategy is to show that the sum of the bubbles at $P$ and the sum of the bubbles at $Q$ are in the \textit{same homology class} in $H_2(\Nc, \Z)$. When this will be achieved, we will conlcude by means of a \textit{calibration argument}. Recall that on $\Nc$ (with its metric $h$) we have an $S^2$-family of calibrations, i.e. the K\"ahler forms $\alpha_{(a,b,c)}:=a \Om_I + b \Om_J + c\Om_K$ for any choice of $(a,b,c)$ with $a^2+b^2+c^2=1$ and a $2$-dimensional integral current is said to be calibrated by $\alpha_{(a,b,c)}$ when a.e. approximate oriented tangent plane $\vec{T}$ is such that the calibration $\alpha_{(a,b,c)}$ and the natural area element induced from $h$ have the same action on $\vec{T}$. For the K\"ahler forms the notion of being calibrated by  $\alpha_{(a,b,c)}$ is equivalent to being holomorphic with respect to the associated complex structure $a I + b J + c K$ (see e.g. \cite[page 4]{B}).

\medskip

We need to recall at this stage, from Section \ref{quantization}, how the analysis of the bubbling at $P$ was performed. Recall in particular the choice of the sequence $X_1^\ell$, see (\ref{eq:choice3}). This choice was performed on the sequence $\tilde{u}_\ell$, that had been obtained in the beginning of Section \ref{quantization} by suitable dilations $(u_\ell)_{P, \lambda_\ell}$ of the given sequence $u_\ell$ at $P$. We focused on $2$-planes normal to $T_P \Sigma$ (which was identified with $B^{4m-2} \times (0,0)$ by a suitable coordinate choice). Fix any $\ell$ large enough and let us read the $2$-planes normal to $T_P \Sigma$ back in the picture of $u_\ell$:  we find a foliation (whose leaves are $2$-disks) of $B_{\lambda_\ell}(P)$. These disks are almost normal to $\mathcal{L} \cap B_{\lambda_\ell}(P)$. The same goes for $Q$, so for each $\ell$ we have a foliation (whose leaves are $2$-disks) of $B_{\lambda_\ell}(Q)$. The disks, as $\ell$ gets large, are actually staying the same, the only thing happening as $\ell$ grows is that we are only seeing the part of the disks inside the smaller balls $B_{\lambda_\ell}(P)$ and $B_{\lambda_\ell}(Q)$.

For $\ell$ large enough the disks in $B_{\lambda_\ell}(P)$ and $B_{\lambda_\ell}(Q)$ are small compared to the size of $B^{4m}$. We then consider a closed $2$-dimensional surface, diffeomorphic to a sphere, that intersects $\Sigma$ exactly at $P$ and $Q$ (and nowehere else - here we are using the fact that $\Sigma$ is contained in a Lipschitz graph). This surface should be slightly deformed around $P$ and $Q$ so that it is $C^2$ and it contains the $2$-disk passing through $P$ and the $2$-disk passing through $Q$ (for a certain $\ell$ large enough, and thus for all larger $\ell$); this is possible by the Lipschitz condition and the differentiability at $P$ and $Q$. We will denote this surface by $S$ and still call it a sphere by abuse of terminology. Without loss of generality we can give an orientation to $\mathcal{L} \cap B^{4m}$ and $S$ and assume that the intersection is positive at $P$ and negative at $Q$. Moreover we can consider a tubular neighbourhood of $S$ and we foliate it by similar spheres, each of which contains two disks of the disk-foliation. For $\ell$ large enough all the $B_{\lambda_\ell}(P)$ and $B_{\lambda_\ell}(Q)$ are contained in the tubular neighbourhood of the first $S$. Each $2$-disk identifies uniquely a sphere (recall that the $2$-disks correspond to choices of $X^1$ in Section \ref{quantization}).

From this $(4m-2)$-parameter foliation with sphere-leaves we want to select only ``good spheres'', meaning the following. Recall the choice of $X_1^\ell$ in Section \ref{quantization} and choose the corresponding sequence $S_\ell$ of spheres from the foliation. Keep in mind that in the choice of $X^1_\ell$ we had a huge freedom (we only needed to avoid a $\Hc^{4m-2}$-small set of values) so the same freedom applies here: we only need to avoid, for each $\ell$, a small set of elements in the foliation, where smallness is understood in the $\Hc^{4m-2}$-measure with respect to $\lambda_\ell^{4m-2}$ (these are bad spheres). Observe moreover that (up to neglecting a $\Hc^{4m-2}$-negligeable set of elements in the foliation, other bad spheres) we can ensure that for every $\ell$ the restriction $u_\ell|_{S_\ell}$ is smooth (by the regularity of the maps $u_\ell$, see Proposition \ref{Prop:singset}, and the fact that there are countably many $u_\ell$). 
% The last $\Hc^{4m-2}$-negligeable set of bad spheres is thrown away by the following observation: the condition (\ref{eq:strongapproximability}) gives that on $B^{4m}$ we have smooth maps $f_k^\ell$ converging strongly in $W^{1,2}$ to $u_\ell$ as $k \to \infty$; this strong convergence is valid also for the restrictions to good spheres of the foliation (this is true for $\Hc^{4m-2}$-a.e. choice of spheres, by the coarea formula).

Clearly the spheres $S_\ell$ tend to the initial $S$ (the one though $P$ and $Q$). By the parametrization of the tubular neighbourhood of $S$ with domain $S \times (-\varepsilon, \varepsilon)$ we can also identify each $S_\ell$ with $S$, which will be implicitly understood in (\ref{eq:strongconvbase}). 

\medskip

By the blow-up analysis from Section \ref{blowupanalysis} we have that $u_\ell \to u$ in $C^k$ away from any neighbourhood of $\Sigma$. For every fixed $\overline{\ell}$ large enough, we then have two small $2$-dimensional disks $D(P_{\overline{\ell}})$ and $D(Q_{\overline{\ell}})$ in $S_{\overline{\ell}}$ (the disks from the disk-foliation, we can assume that they are small compared to the size of $S_{\overline{\ell}}$) such that 

$$u_\ell|_{S_{\overline{\ell}} \setminus (D(P_{\overline{\ell}}) \cup D(Q_{\overline{\ell}}))} \to u|_{S_{\overline{\ell}} \setminus (D(P_{\overline{\ell}}) \cup D(Q_{\overline{\ell}}))} \text{ in }C^k.$$

Remark further that, as $\overline{\ell} \to \infty$, the disks $D(P_{\overline{\ell}})$ and $D(Q_{\overline{\ell}})$ are slighly translating and shrinking, converging (in any reasonable sense) respectively to $P$ and $Q$. Moreover $u$ is smooth, therefore we have that 

\begin{equation}
 \label{eq:strongconvbase}
u_\ell|_{S_{\ell} \setminus (D(P_{\ell}) \cup D(Q_{\ell}))} \to u|_{S \setminus \{P,Q\}} \text{ in } C^{k}_{loc} \text{(we will need uniform convergence only)}.
\end{equation}
$$u_\ell|_{D(P_{\ell})} \to u(P),\,\,\,\,\,\, u_\ell|_{D(Q_{\ell})} \to u(Q).$$

\medskip

We have implicitly identified $S_\ell$ with $S$ in this convergence result. Moreover we know from Section \ref{quantization} that for every $\ell$ large enough we can identify inside $D(P_\ell)$, with a finite iterative procedure, a finite number of \textit{bubble domains} (identified by means of a maximal function) in which the energy concentrates: here we can renormalise, i.e. reparametrize $u_\ell$ by suitable dilations and stereographic projections. The basic step consists of dilating $D(P_\ell)$ (more and more as $\ell \to \infty$) and wrapping it on an $S^2$, and this reparametrization leads to the first bubble map $\phi_1$. For each $\ell$ the domain of the bubble that we are forming is $S^2 \setminus \{\text{small disk}\}$ (and this small disks shrinks to a point as $\ell \to \infty$). For each $\ell$ we replace the small disk $D(P_\ell)$ with this excised $S^2$ (domain of $\phi_1$), see the top-left picture in Figure \ref{fig:bubbletree}. On the spherical excised bubble domain the (reparametrized) map converges in $C^{1, \alpha}$ away from a possibly problematic neck, on which we continue the contruction after transforming it conformally into a long cylinder (using (\ref{eq:elongatedneck}), see figure \ref{fig:bubbletree} top-right). Sometimes we will need to perform this basic step on a neck domain (which could have a bubble domain in it)\footnote{Comparing to Section \ref{quantization} this corresponds to the case when we have weak convergence of $V_\ell$ to $V_\infty$.}. Again this leads to replacing a small disk in the neck with an excised $S^2$ attached by another neck (Figure \ref{fig:bubbletree}-bottom right). The other possibility that could happen is that on a neck we have the formation of a bubble without the need of a dilation, in which case we replace a cylinder of fixed size by a $S^2$ with two disks removed\footnote{In Section \ref{quantization} this corresponds to the case when we have strong convergence of $V_\ell$ to $V_\infty$.}, compare Figure \ref{fig:bubbletree}-bottom left picture.
After a finite number of steps of this type we exhaust the bubbles and reparametrize on every bubble domain. This leads to the replacement of $S_\ell$ by a new domain, that is topologically still a sphere, but looks like a finite number of spherical excised bubble domains $S_{1,\ell}^2$, ..., $S_{S_P,\ell}^2$ connected by cyclindrical necks with $S_\ell \setminus D(P_\ell)$.

\begin{figure}[!hb]
\centering
 \includegraphics[width=9cm]{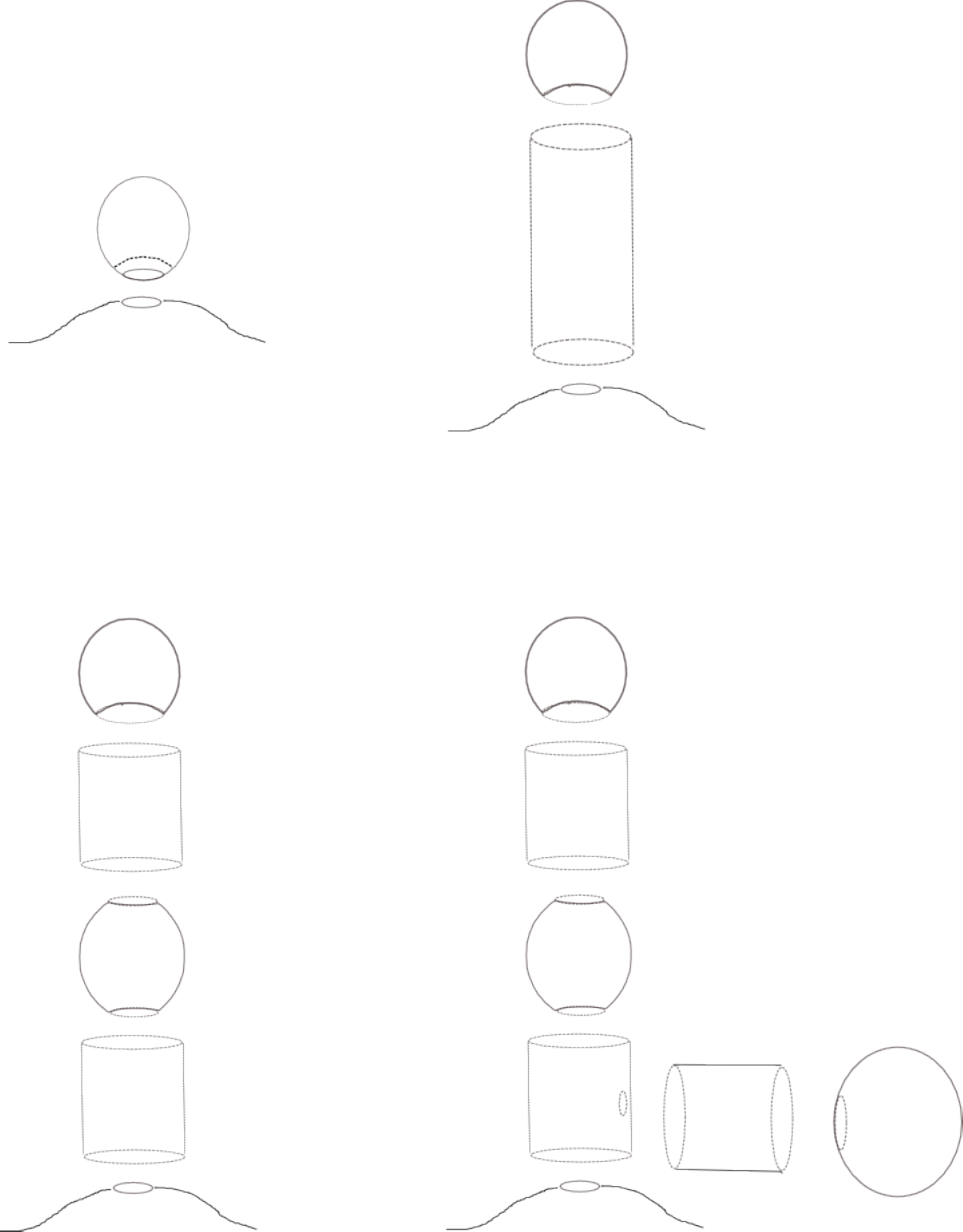}
\caption{This example illustrates the passage from $S_\ell$ to $\tilde{S}_\ell$ at a point of $\Sigma$ (say $P$) in the case that there are three bubbles arising. In the top-left picture we are producing the first bubble $\phi_1$. On the part of the sphere above the dotted line we have $C^{1,\alpha}$-convergence to the bubble map $\phi_1$, the neck (annulus) below the dotted line gets conformally transformed ino a long cylinder and we continue on it the bubbling analysis (top-right picture). We are assuming that on this cylinder there is (in the centre) convergence to a bubble $\phi_2$ without need of reparametrization (this corresponds to $V_\ell \to V_\infty$ strongly in $W^{1,2}$ with the notations of Section \ref{quantization}). $V_\infty$ becomes the second bubble $\phi_2$ and we visualize this by replacing the middle part of the cyclinder with a spherical bubble domain (bottom-left picture). Here we see two necks and since there is only another bubble to be tracked down, one of the two necks allows no further bubble domains (say the top neck). In the bottom neck we assume that there is a bubble formation but it needs a reparametrization (this corresponds to $V_\ell \to V_\infty$ weakly in $W^{1,2}$ with the notations of Section \ref{quantization}). Thus we again replace a small cap with a bubble domain attached by a neck (bottom-right picture) and produce the third bubble $\phi_3$, by repeating the procedure that we used for $\phi_1$.}
 \label{fig:bubbletree}
\end{figure}

% of factor $\delta_{a,\ell}$ and define a new map $\tilde{u}_{a,\ell}:D_{a,\ell} \to \Nc$ defined on a $2$-disk with radius $\frac{r_a}{\delta_{a,\ell}}$. The new maps, defined on larger and larger disks with radii going to $\infty$, converge in $C^{1,\alpha}_{\text{loc}}$ to a smooth harmonic map $v_j$ from $\R^2$ into $\Nc$ as in (\ref{eq:convergenceofthefirstbubble}). By a conformal change of coordinates (using Sacks-Uhlenbeck's result) we can replace the limiting domain $\R^2$ of $v_j$ by an $S_{B_j}^2$ attached to center of the disk $B_j$. In the same way we replace the larger and larger disks that exhaust $\R^2$ with $S_{B_j}^2 \setminus \text{ (small disk)}$ and paste the boundary of this to $\p B_j$. By doing so for $j=1,...,S_P$ we reparametrize $u_\ell|_{D(P_\ell)}$ by replacing the domain $D(P_\ell)$ with a connected sum of $S_P$ spheres, attached by the \textit{neck domains}. Topologically this new domain (that replaces $D(P_\ell)$) is still a disk and we can attach it to $S_\ell \setminus D(P_\ell)$ to form a two-dimensional sphere once again. 
For $D(Q_\ell)$ we do the same thing and replace it with a connected sum of $S_Q$ spheres, attached by the \textit{neck domains}. So $S_\ell$ has been replaced by a new domain $\tilde{S}_\ell$ which is topologically still a sphere but, thanks to the renormalization process, the new map $\tilde{u}_\ell$ defined on it by reparametrization can be passed to the limit in a much stronger sense. Observe also that $\tilde{u}_\ell$ is still smooth, since so was $u_\ell$. Remark that we have strong convergence on each spherical bubble domain to a smooth harmonic map, strong convergence to $u$ on $S_\ell \setminus D(P_\ell)$ as in (\ref{eq:strongconvbase}) and the contribution to the Dirichlet energy in the necks is as small as we like for $\ell$ large enough (the latter is the content of Theorem \ref{thm:quantization}). 

\medskip

To be precise for the forthcoming argument, we need to fill in the missing disks of the excised spheres and cylinders, i.e. we want to think of $\tilde{S}_\ell$ as a connected sum of spheres (the cylinders also become spheres once we fill in the missing disks):
$$\tilde{S}_\ell =  S_\ell \# \sum_j N_{\ell, j} \# \sum_k S^2_{\ell, k} $$
with $j$ and $k$ ranging over a finite set. Here we denote by $N$ the spheres that came from the neck domains and by $S^2$ those that came from the bubble domains. Roughly speaking we are thinking of $\tilde{S_\ell}$ as a sum of closed spheres and closed cyclinders, attached by small disks: when we attach them by overlapping the small disks, orientation cancellation gives back $\tilde{S}_\ell$.

It is not essential to know exactly how we extend the map $\tilde{u}$ on these extra disks. The disks can be taken as small as needed, up to chosing $\ell$ large enough; by the strong convergence on each bubble domain to a smooth harmonic map (a bubble map) and by the strong convergence on $S_\ell \setminus (D^P_{\ell,1} \cup D^Q_{\ell,1})$ to the smooth harmonic map $u$ we can define $\tilde{u}$ on these small disks by using as little Dirichlet energy as we like\footnote{We are counting the extra energy brought in by closing the spheres and the cylinders, here each disks contributes to the energy twice as there is no cancellation.}, say $\tilde{\eps}$ in total. The area covered in $\Nc$ by the image of these small disks is also at most $\tilde{\eps}$, since the Dirichlet energy controls the area.

Recall that by Hurewicz theorem $\pi_2(\Nc)$ is the same as $H_2(\Nc, \mathbb{Z})$. Consider now the map $\tilde{u}_\ell: \tilde{S}_\ell \to \Nc$ and remark that $(\tilde{u}_\ell)_*(\llbracket\tilde{S}_\ell\rrbracket)$ is the same integral $2$-dimensional cycle as $(u_\ell)_*(\llbracket S_\ell \rrbracket)$. We want to look at the homology class of $(u_\ell)_*(\llbracket S_\ell \rrbracket)$. This class can be obtained by adding the homology classes of the spherical sub-domains $S^2_{\ell, k}$ (the bubble domains) and those of the cyclindrical necks $N_{\ell, j}$ (these are now all closed surfaces, i.e. topologically they are $S^2$'s but we will keep the names to distinguish the bubble domains and the neck domains).

\textbf{Claim}: the homology class of $(\tilde{u}_\ell)_*(\llbracket\tilde{S}_\ell\rrbracket)$ is the sum of the class of $u: S \to \Nc$ plus the classes of  $v_p$ (the bubbles originating at $P$, here $p=1,..., S_P$) minus the classes of $w_q$ (the bubbles originating at $Q$, here $q=1,..., S_Q$-recall that at $Q$ we have that $D(Q)$ is oriented negatively):
$$(\tilde{u}_\ell)_*\llbracket\tilde{S}_\ell\rrbracket\stackrel{\text{in }H_2(\Nc,\Z)}{\equiv}(u)_*\llbracket{S}\rrbracket + \sum_{p=1}^{S_P} (v_p)_*\llbracket S^2 \rrbracket -  \sum_{q=1}^{S_Q} (w_q)_*\llbracket S^2 \rrbracket. $$

This claim follows by the following facts:
\begin{description}
 \item[(i)]  we have $C^{1,\alpha}$-convergence on the bubble domains to the bubble maps $v_p$ and $w_q$ (thanks to the reparametrization), therefore we can pass to the limit in the homotopy classes; 
 
 \item[(ii)] we have the uniform convergence of $u_\ell|_{S_\ell}$ to $u|_S$ as in (\ref{eq:strongconvbase}), therefore we can pass to the limit in the homotopy classes;  
 
 \item[(iii)] the energy in the cylindrical neck domains goes to $0$ by Theorem \ref{thm:quantization} and a fortiori for each neck domain the area of its image goes to $0$ as $\ell \to \infty$, therefore the (finitely many) necks do not contribute in $H_2(\Nc, \Z)$. Recall that for each non-zero homology class there exists a mass-minimizer with non-zero mass and there exists $\xi>0$ such that any integral $2$-cycle with mass below $\xi$ must be in the trivial class, see \cite{FF}. Here we are also ensuring that the $\tilde{\eps}$ of extra energy that we introduced by closing up the cylinders is small compared to $\xi$, as we said this is no problem as long as $\ell$ is large enough.
\end{description}

With the claim in mind, observe further that the map $u$ is smooth on $B^{4m}$, which means in particular that $u: S \to \Nc$ is contractible, so its class is $[0] \in \pi_2(\Nc)$. Therefore
\begin{equation}
 \label{eq:homologyclassprimo}
(\tilde{u}_\ell)_*\llbracket\tilde{S}_\ell\rrbracket\stackrel{\text{in }H_2(\Nc,\Z)}{\equiv}\sum_{p=1}^{S_P} (v_p)_*\llbracket S^2 \rrbracket -  \sum_{q=1}^{S_Q} (w_q)_*\llbracket S^2 \rrbracket. 
\end{equation}
Moreover the assumption that $u_\ell$ satisfy (\ref{eq:strongapproximability}) yields as well that $(u_\ell)_*(\llbracket S_\ell \rrbracket)$ is in the trivial class of $H_2(\Nc,\Z)$, as follows. Denote the bubbles at $P$ with $v_p:S^2 \to \Nc$ and the bubbles at $Q$ with $w_q:S^2 \to \Nc$, for $p$ ranging in $\{1, ... S_P\}$ and $q$ ranging in $\{1, ... S_Q\}$ (probably $S_P$ must be the same as $S_Q$ but this is not important for this proof). 
The bubbles at $P$ are all holomorphic for the complex structure $a_P I + b_P J + c_P K$ whilst the bubbles at $Q$ are all holomorphic for the complex structure $a_Q I + b_Q J + c_Q K$. For any closed $2$-form $\alpha=a \Om_I + b \Om_J +c \Om_K$ on $\Nc$ with $a^2+b^2+c^2=1$ we consider $(u_\ell)_*(\llbracket S_\ell \rrbracket)(\alpha) = \llbracket S_\ell \rrbracket(u_\ell^* \alpha)$ and keep in mind that $\llbracket S_\ell \rrbracket = \p L$ for some $3$-current $L$ (since we are in a ball $B^{4m}$), therefore $(u_\ell)_*(\llbracket S_\ell \rrbracket)(\alpha) = L(d(u_\ell^* \alpha)) =0$ by (\ref{eq:strongapproximability}). This is true for all the above choices of $\alpha$. Unless the bubbles at $P$ and the bubbles at $Q$ are holomorphic for the same complex structure (which is the desired conclusion of the theorem and in which case we can show even more easily that $(u_\ell)_*(\llbracket S_\ell \rrbracket)$ is in the trivial class) we can choose $\alpha$ above so that its integration on $ \sum_{q=1}^{S_Q} (w_q)_*\llbracket S^2 \rrbracket$ gives $0$ and its integration on $\sum_{p=1}^{S_P} (v_p)_*\llbracket S^2 \rrbracket$ is non-zero \footnote{There is an $S^2$ of complex structures on $\Nc$ and so there is an $S^1$ of complex structures orthogonal to $a_Q I + b_Q J + c_Q K$ and exactly one complex structure that is orthogonal to both $a_P I + b_P J + c_P K$ and $a_Q I + b_Q J + c_Q K$.}. Since $\alpha$ is closed, its integration gives the same result within a fixed homology class, so we have contradicted (\ref{eq:homologyclassprimo}).\footnote{If we assume that on $B^{4m}$ we have smooth maps $f_k^\ell$ converging strongly in $W^{1,2}$ to $u_\ell$ as $k \to \infty$, then up to neglecting a $\Hc^{4m-2}$-negligeable set of further bad spheres in the foliation, this strong convergence is valid also for the restrictions to good spheres $S_\ell$ (by the coarea formula).  Therefore the currents $(f_k^\ell)_*\llbracket S \rrbracket$ converge (weakly as integral currents) to $(u_\ell)_*\llbracket S \rrbracket$. By the smoothness of $f_k^\ell$ the maps $f_k^\ell: S \to \Nc$ are contractible and so the currents $(f_k^\ell)_*\llbracket S \rrbracket$ are in the trivial homology class of $H_2(\Nc, \Z)$ and this passes to the limit by the compactness theorem of \cite{FF}.} Therefore the claim is strenghtned to
\begin{equation}
 \label{eq:samehomologyclass}
\sum_{p=1}^{S_P} (v_p)_*\llbracket S^2 \rrbracket -  \sum_{q=1}^{S_Q} (w_q)_*\llbracket S^2 \rrbracket \stackrel{\text{in }H_2(\Nc,\Z)}{\equiv}0. 
\end{equation}
We will conlcude by showing that (\ref{eq:samehomologyclass}) can only happen if $a_P I + b_P J + c_P K=a_Q I + b_Q J + c_Q K$, thereby proving the result.\footnote{Remark, as a comparison with the general case of bubbling for stationary harmonic maps, that the class of the sum of the bubbles at $P$ is non-trivial, as we noticed in remark \ref{oss:holo&antiholo}: there can be no cancellation in homology between the bubbles at $P$ (the same happens for $Q$).}Consider the integral $2$-current 
$$C_P:=\sum_{p=1}^{S_P} (v_p)_*(\llbracket S^2\rrbracket).$$
As a sum of holomorphic spheres it is calibrated by the $2$-form $a_P \Om_I + b_P \Om_J + c_P \Om_K$ (this is a calibration of degree $2$ in $\Nc$), i.e. the tangent $2$-planes are of the form $\vec{v} \wedge (a_P I + b_P J + c_P K)\vec{v}$. Then the standard calibration argument shows that $C_P$ miminizes the mass $M$ (i.e. the area counting multiplicities) in its homology class: indeed, let $T$ be another $2$-current in the same homology class, we then have $T - C_P = \p L$ for some $3$-current $L$ in $\Nc$ and we can write, using the closedness of $a_P \Om_I + b_P \Om_J + c_P \Om_K$ in the second equality and Wirtinger's theorem in the last inequality (equivalently, the fact that the comass is $1$):

$$M(C_P)=C_P(a_P \Om_I + b_P \Om_J + c_P \Om_K) = T(a_P \Om_I + b_P \Om_J + c_P \Om_K) \leq M(T). $$
The amount of information of this simple computation is even greater: equality in the last step holds if and only if $T$ as well is calibrated by $a_P \Om_I + b_P \Om_J + c_P \Om_K$. We conclude terefore that any minimizer in the homology class of $C_P$ must be calibrated by $a_P \Om_I + b_P \Om_J + c_P \Om_K$.

Consider now the integral $2$-current 
$$C_Q:=\sum_{q=1}^{S_Q} (w_q)_*(\llbracket S^2\rrbracket).$$
This is calibrated by $a_Q \Om_I + b_Q \Om_J + c_Q \Om_K$ and therefore (as before) we have that it is a mass-minimizer in its homology class. We concluded in (\ref{eq:samehomologyclass}) that $C_P$ and $C_Q$ are in the same homology class. Thus $C_Q$ must be calibrated by $a_P \Om_I + b_P \Om_J + c_P \Om_K$, i.e. $a_P=a_Q$, $b_P=b_Q$, $c_P=c_Q$.

Observe also that $\Theta(P)$ is given by the sum of the energies of the bubbles at $P$ by Theorem \ref{thm:quantization}, and for each bubble the Dirichlet energy agrees with the mass of the induced current (by conformality). We have seen that $M(C_P)=M(C_Q)$ and therefore $\Theta(P)=\Theta(Q)$. 

As we remarked in the beginning, $P$ and $Q$ can be chosen in an arbitrary manner on $\Sigma \cap B^{4m}$ among points where the density of $\Sigma$ is $1$ and the approximate tangent to $\Sigma$ exists and agrees with the tangent at $\mathcal{L}$. In particular we have proved that almost all points in $\Sigma \cap B^{4m}$ have an approximate tangent for a unique almost complex structure, but we still do not know if $\mathcal{L} \setminus \Sigma$ is empty or not. To show that it is empty, and thus conlcude the proof of Theorem \ref{thm:localholomorphicitysmooth}, we take two points $P$ and $Q'$ on $\mathcal{L}$ with $P \in \Sigma$ and $Q'$ in the open set $\mathcal{L} \setminus \Sigma$, assuming again that $P$ is a point where $T_P\Sigma$ exists. Then we repeat the argument given in this section. This time the quantization result Theorem \ref{thm:quantization} is used only at $P$, while at $Q$ we have no bubbling. Therefore, instead of (\ref{eq:samehomologyclass}) we will conlcude that the sum of the bubbles at $P$ is trivial in $H_2(\Nc, \Z)$, contradicting that $P$ was a point of bubbling. We have thus concluded that $\Sigma \cap B^{4m} = \mathcal{L}$ and $\mathcal{L}$ is a pseudo holomorphic Lipschitzian graph.

\medskip

As remarked in the statement of Theorem \ref{thm:localholomorphicitysmooth}, a posteriori we have that $\mathcal{L}$ is actually smooth. Indeed first of all we observe that $\frac{1}{\Theta}\llbracket \mathcal{L} \rrbracket$ is an integral $(4m-2)$-cycle that is semi-calibrated by the form $\frac{1}{(2m-1)!}(a \om_i + b \om_j +c \om_k)^{(2m-1)}$, and thus is has bounded first variation (we refer to \cite{B2} for general semi-calibrations and in particular for those semi-calibrations obtained by powers of a two-form). Its multiplicity is $1$ almost everywhere and so we can improve its regularity to $C^{1,\alpha}$ by Allard's theorem \cite{Allard}. Then we can improve it to $C^{\infty}$.

More directly (even without recalling Allard's result), we can write $\mathcal{L}$ as the graph of $f:B^{4m-2}\subset \R^{4m-2} \to \R^2$ in suitable coordinates. Namely, we can produce a map $\psi: D^{4m-2} \times D^2 \to \Mc$ so that for every $p \in D^{4m-2}$ the $2$-disk $\psi(p, \cdot)$ is an embedded pseudo holomorphic disk with respect to the almost complex structure $ai + b j + ck$ and so that this family of disks foliates $B^{4m}$. Such foliation is produced in \cite{B}. We can then express $\mathcal{L}$ as the graph of a $W^{1,\infty}$ function $f:\mathbb{D}^{2m-1}\subset \C^{2m-1} \to \C$  in the coordinates induced by $\psi$, with $f(0)=0$ and $\nabla f(0)=0$. Imposing the condition of $J$-invariance on its (a.e. well-defined) tangent planes, we obtain a perturbation of the Cauchy-Riemann equations of the type

$${\overline{\p}}_j f = A^s_j(\vec{z}, f(\vec{z})) \nabla_j f + B_j(\vec{z}, f(\vec{z})),$$
where $\vec{z}=(z_1, ..., z_{2m-1})$ and $A^s_j$ and $B_j$ are smooth functions that are $0$ at $(0,0)$ and are small in $C^k$-norm. These two perturbations functions do not have any dependence on $\nabla f$ in the arguments, therefore this PDE can be treated as a perturbation of Cauchy-Riemann by means of standard elliptic theory, and allows bootstrapping to $C^{\infty}$. This strategy is an easy case of the one employed in \cite{RT1}, where a given pseudo holomorphic integral $2$-cycle in an almost complex manifold of dimension $4$ is studied and its representation in such coordinates is used to deal with the harder case of high multiplicities. 

\medskip

For the sake of completeness, remark that the smoothness also implies that the bubbling conclusions are actually true everywhere on $\Sigma \cap B^{4m}$, since all points admit a tangent plane $T_x \Sigma$.

\section{An application to Fueter sections}
\label{Fueter}

When the domain manifold is $4$-dimensional ($m=1$) the triholomorphic map equation (\ref{eq:triholo}) corresponds to the classical quaternionic $\p$-bar equation studied for the first time in \cite{Fueter} and thus also called the \textit{Fueter equation}. In this section we describe a rather direct application of our work to the ``bubbling analysis for Fueter sections'', treated by T. Walpuski \cite{Walpuski}. The setting in \cite{Walpuski} requires a slightly more general notion than the one we have taken in the present paper, namely one wants to deal with a domain manifold $\Mc^4$ that carries an almost hyper-Hermitian structure only locally (this can be always achieved on a $4$-dimensional manifold) but not necessarily globally, since a topological obstruction might be present, preventing the global extension of $(i,j,k)$. In order to overcome this, one considers a compact bundle $\mathcal{\chi}$ of hyperK\"ahler manifolds on $\Mc$ with a fixed identification of the unit sphere bundle of self-dual forms on $\Mc$ with the bundle of hyperK\"ahler spheres of the fibres of $\chi$. Moreover one needs to fix a connection $1$-form $A$ on $\chi$. To clarify ideas, observe that in the case of triholomorphic maps studied in the present work, we would have the trivial bundle $\Mc \times \Nc$ endowed with the flat connection and the trivial identification $\om_i \to I$, $\om_j \to J$, $\om_k \to K$. In the bundle one then looks at a (smooth) Fueter section, i.e. $u \in \Gamma(\chi)$ that satisfies \cite[(1.2) or (B.3)]{Walpuski}. We remark, in order to avoid confusion, that the situation studied in the main body of \cite{Walpuski} would correspond in our situation to a homogeneous triholomorphic map (a so-called \textit{tangent map}), which is why the domain is taken $3$-dimensional; in Appendix B \cite{Walpuski} deals with the $4$-dimensional case.
As explained in \cite{Walpuski}, the study of Fueter sections is motivated by Gauge Theory on $G_2$ and $Spin(7)$–manifolds, in relation with codimension four bubbling phenomena. The bundle in question is, in that case, a bundle of moduli spaces of anti self-dual instantons on $\Mc$, which can be given a hyperK\"ahler structure.

The presence of the connection changes our equation (\ref{eq:triholo}) by adding on the right hand side lower order terms (depending on $u$ but not on its derivatives). Indeed, compare equation (1.2) or (B.3) in \cite{Walpuski}, the term $Idu i$ in (\ref{eq:triholo}) must be replaced with a term of the form $I (d+A) i$, where we are writing $(d+A)$ for the covariant derivative induced by the chosen connection. In local coordinates, with summation over repeated indexes, this term reads
$$I \nabla_{i\frac{\p}{\p x^\ell}}u= I du\left(i\frac{\p}{\p x^\ell}\right) +I(u) A (u)\lrcorner\left(i\frac{\p}{\p x^\ell}\right)=I_\alpha^\beta(u) i_\ell^s \frac{\p u^\alpha}{\p x^s}\frac{\p}{\p y^\beta}+I_\alpha^\beta(u) A_s^\alpha(u) i_\ell^s \frac{\p}{\p y^\beta}.$$
Remark that in the last two terms, the first one is just what appeared in (\ref{eq:triholoincoords}), so the perturbation term in (\ref{eq:triholo}) is the second term (similar terms appear for $(j,J)$ and $(k,K)$). We will describe briefly how such terms do not affect the analysis that we have performed in order to reach our main results.

In Section \ref{jacobianstructure} we used the triholomorphic equation (\ref{eq:triholo}) together with the geometric structures on the manifolds $\Mc$ and $\Nc$ to infer from (\ref{eq:triholoincoords}), in local coordinates on $\Mc$, the second order PDE  (\ref{eq:withskewsymmetry}), which shows the jacobian structure for $\Delta_g u$. This was the key initial step for the $\eps$-regularity result and for the quantization result. In the situation of Fueter sections from \cite{Walpuski} the perturbation terms do not affect our analysis: indeed, using our notations, these terms reflect in an extra $L^2$-term on the right hand side of (\ref{eq:withskewsymmetry}) when we compute $\Delta_g u^\beta$, namely the perturbation is

$$\frac{1}{\sqrt{g}} \frac{\p}{\p x^a} \left( \sqrt{g} g^{a \ell}  I_\alpha^\beta(u) i_\ell^s A_s^\alpha(u) + \sqrt{g} g^{a \ell}  J_\alpha^\beta(u) j_\ell^s A_s^\alpha(u)+ \sqrt{g} g^{a \ell}  K_\alpha^\beta(u) k_\ell^s A_s^\alpha(u)\right)$$
and the $L^2$-norm of these extra terms is controlled up to a universal constant by the Dirichlet energy of $u$ (only first derivatives of $u$ appear and we never take products of derivatives of $u$). In other words, the PDE structure $\Delta_g u \in \mathscr{h}^1$ that we produce in Proposition \ref{Prop:JacobianStructure}, on which our later analysis is ultimately based, is preserved with the same estimate $\|\Delta_g u\|_{\mathscr{h}^1(\Uc)} \leq C \int_{\Uc}|\nabla u|^2$ (up to changing the constant appearing on the r.h.s.). Our analysis can then be carried out in the same fashion in the application under consideration. The $\eps$-regularity result for Fueter sections was established independently in \cite{Walpuski} and we could also adapt the proof from Section \ref{epsreg}. In adddition we can give an affirmative answer to the quantization question, which is conjectured to be true but left open in \cite[page 5]{Walpuski} and show the analogue of our result Theorem \ref{thm:localholomorphicitysmooth}.

\begin{thm}
Let $u_\ell \in \Gamma(\chi)$ be a sequence of Fueter sections, i.e. they satisfy (1.2) in \cite{Walpuski} and let us assume that the Dirichlet energies $E(u_\ell)$ are uniformly bounded. Up to a subsequence the $u_\ell$ converge weakly in $W^{1,2}$ to a section $u$ (possibly singular) with associated rectifiable defect measure $\Theta \mathcal{H}^1  \res S$ (we are using notations as in \cite[Thm. 1.9]{Walpuski}). Then there is an a priori $W^{2,1}$-estimate
$$\|u_\ell\|_{W^{2,1}}\leq C E(u_\ell), $$
for a universal constant $C>0$. Moreover for $\mathcal{H}^1$-a.e. $x \in S$ there exist non-constant holomorphic spheres (the bubbles at $x$) $\zeta_s: S^2 \to \chi_x$, where $\chi_x$ is the fiber on $x \in \Mc$ and where $S_x \in \N$ and $s=1, ..., S_x$, so that the following energy identity holds
$$\Theta(x)= \sum_{s=1}^{S_x} E(\zeta_s).$$
The bubbles at $x$ are a priori all holomorphic for a complex structure that depends on $x$. Whenever $\Sigma$ is (locally in an open ball $B^3$) contained in a connected Lipschitz graph $\mathcal{L}$ (without boundary) and if the limit section $u$ is smooth on $B^3$ then for all $x \in \Sigma \cap B^3$ the bubbles are holomorphic for a unique complex structure (independent of $x$) and $\Sigma \cap B$ is a smooth connected curve in $B^3$.
\end{thm}

\end{document}